\documentclass[reqno]{amsart}

\usepackage{amsmath}
\usepackage{amssymb}
\usepackage{amsfonts}
\usepackage{amsthm}
\usepackage{longtable}
\usepackage[hidelinks]{hyperref}
\usepackage{tikz-cd}
\usepackage{moresize}

\newcommand{\N}{\mathbb{N}} 
\newcommand{\Z}{\mathbb{Z}}
\newcommand{\R}{\mathbb{R}} 
\newcommand{\C}{\mathbb{C}}

\newcommand{\D}{\mathcal{D}}
\newcommand{\cO}{\mathcal{O}}
\newcommand{\cP}{\mathcal{P}}

\newcommand{\g}{\mathfrak{g}}
\newcommand{\q}{\mathfrak{q}}
\newcommand{\h}{\mathfrak{h}}
\newcommand{\fsl}{\mathfrak{sl}}
\newcommand{\fk}{\mathfrak{k}}
\newcommand{\ft}{\mathfrak{t}}

\newcommand{\SO}{\mathrm{SO}}
\newcommand{\Lie}{\mathrm{Lie}}
\newcommand{\e}{\varepsilon}
\newcommand{\s}{\subseteq}
\newcommand{\ip}[1]{\langle #1\rangle}
\newcommand{\too}{\longrightarrow}
\newcommand{\mto}{\mapsto}
\newcommand{\mtoo}{\longmapsto}
\newcommand{\hooklongrightarrow}{\lhook\joinrel\longrightarrow}

\DeclareMathOperator{\Span}{span}
\DeclareMathOperator{\Ad}{Ad}
\DeclareMathOperator{\rk}{rk}
\DeclareMathOperator{\Spec}{Spec}
\DeclareMathOperator{\Hom}{Hom}

\newcommand{\reg}{^{\mathrm{reg}}}
\newcommand{\st}{^{\mathrm{s}}}
\newcommand{\ps}{^{\mathrm{ps}}}
\newcommand{\ass}{^{\textrm{a-ss}}}
\newcommand{\aps}{^{\textrm{a-ps}}}

\newcommand{\sll}[1]{\mkern-4mu\mathbin{/\mkern-5mu/}_{\mkern-4mu{#1}}}
\newcommand{\slll}[1]{\mkern-4mu\mathbin{/\mkern-5mu/\mkern-5mu/}_{\mkern-4mu{#1}}}

\theoremstyle{plain}
\newtheorem{theorem}{Theorem}[section]
\newtheorem{lemma}[theorem]{Lemma}
\newtheorem{proposition}[theorem]{Proposition}
\newtheorem{corollary}[theorem]{Corollary}
\theoremstyle{definition}
\newtheorem{definition}{Definition}[section]

\theoremstyle{remark}

\title{Stratified hyperk\"ahler spaces from semisimple Lie algebras}
\author{Maxence Mayrand}
\thanks{This work was supported by the Moussouris Scholarship and the Fonds de Recherche du Qu\'ebec - Nature et Technologies.}
\address{Maxence Mayrand\\ Mathematical Institute, Andrew Wiles Building\\ University of Oxford\\ Oxford, OX2 6GG\\ United Kingdom}
\email{maxence.mayrand@maths.ox.ac.uk}

\begin{document}

\begin{abstract}
We study singular hyperk\"ahler quotients of the cotangent bundle of a complex semisimple Lie group as stratified spaces whose strata are hyperk\"ahler. We focus on one particular case where the stratification satisfies the frontier condition and the partial order on the set of strata can be described explicitly by Lie theoretic data.
\end{abstract}

\maketitle

\section{Introduction}

Let $G$ be a complex reductive group, or equivalently, the complexification $K_\C$ of a compact connected Lie group $K$. Kronheimer \cite{kro88} showed that the cotangent bundle $T^*G$ can be endowed with a hyperk\"ahler structure. That is, there exist three complex structures $I_1,I_2,I_3$ on $T^*G$ that are K\"ahler with respect to a common metric and satisfy $I_1I_2I_3=-1$. Moreover, the action of $K\times K$ on $G$ by left and right multiplications lifts to an action on $T^*G$ which preserves the hyperk\"ahler structure and has a hyperk\"ahler moment map, i.e.\ a moment map for each of the three K\"ahler forms \cite[Lemma 2]{dan96}. By the hyperk\"ahler quotient construction \cite[\S3(D)]{hit87}, we can thus produce many examples of hyperk\"ahler manifolds by taking hyperk\"ahler quotients of $T^*G$ by closed subgroups of $K\times K$. 

These quotients, of course, might have singularities. For example, Dancer \cite[pp.\ 88--89]{dan93} showed that the hyperk\"ahler quotient of $T^*\mathrm{SL}(2,\C)$ by $\mathrm{U}(1)\times \mathrm{U}(1)$ at level zero is isomorphic to the $D_2$-surface $x^2-zy^2=z$: a complex surface with two isolated singularities. Nevertheless, hyperk\"ahler quotients by non-free actions can always be partitioned into locally closed manifolds each carrying a hyperk\"ahler structure \cite[Theorem 2.1]{dan97}. This is the orbit type partition and is analogous to Sjamaar--Lerman results on singular symplectic quotients \cite{sja91}. Moreover, since the pieces are locally closed, these partitions have an additional structure: a partial order given by $S\le T$ if $S\s\overline{T}$. The partition is said to satisfy the {\bf frontier condition} if $S\cap\overline{T}\ne\emptyset\Longrightarrow S\le T$. We say that a locally finite partition into locally closed manifolds is a {\bf stratification} if it satisfies the frontier condition. For example, the orbit type partition of a singular symplectic quotient is a stratification, but it is not known if the same is true for hyperk\"ahler quotients.

In this paper, we study certain hyperk\"ahler quotients of $T^*G$ by closed subgroups of $K\times K$ and describe their decomposition into hyperk\"ahler manifolds and the partial order on it. More precisely, we focus on a generalization of Dancer's example above to any complex semisimple group $G$. That is, we consider the hyperk\"ahler quotient of $T^*G$ by $T_K\times T_K$, where $T_K$ is a maximal torus in $K$. In particular, we will show that in this case, the orbit type partition is a stratification. 

Although the definition involves a group $G$, the quotient will be shown to depend only on the underlying Lie algebra $\g$. Thus, we denote it by $\D(\g)$. The structure of the stratification of $\D(\g)$ is particularly rich and can be described combinatorially from the root system $\Phi$ of $\g$. We will show that the set of strata is in bijection with the set of root subsystems of $\Phi$ and the partial order corresponds to inclusions of subsystems. In particular, we can compute the stratification structure of $\D(\g)$ explicitly for any given $\g$ (see \S\ref{aejay646nw} for examples).

The strata can also be described individually. Denote by $\D(\g)_\Psi$ the stratum corresponding to a root subsystem $\Psi\s\Phi$. Then, there is always a top stratum $\D(\g)_{\mathrm{top}}:=\D(\g)_\Phi$ and a bottom stratum $\D(\g)_{\mathrm{bottom}}:=\D(\g)_\emptyset$. We will show that $\D(\g)_{\mathrm{top}}$ is a connected open dense subset of real dimension $4(\dim\g-2\rk\g)$ and $\D(\g)_{\mathrm{bottom}}$ a finite set of $|W|$ elements, where $W$ is the Weyl group of $\g$. Note that this generalizes Dancer's example that $\D(\fsl(2,\C))$ is a complex surface with two isolated singularities. To describe the intermediate strata, let $\g_\Psi$ be a semisimple Lie algebra with root system $\Psi$ and $W_\Psi$ the Weyl group of $\g_\Psi$. Then, for any root subsystem $\Psi\s\Phi$, the stratum $\D(\g)_\Psi$ is isomorphic as a hyperk\"ahler manifold to a disjoint union of $|W|/|W_\Psi|$ copies of $\D(\g_\Psi)_{\mathrm{top}}$.

We will also describe a coarser stratification which keeps the property that the strata are disjoint unions of copies of spaces $\D(\g_\Psi)_{\mathrm{top}}$. In this case, the set of strata is in bijection with the set of root subsystems modulo the Weyl group action, or equivalently, the set of conjugacy classes of regular semisimple subalgebras of $\g$.

To obtain these results, we will first prove a Kempf--Ness type theorem for the hyperk\"ahler quotient of $T^*G$ by any closed subgroup $H$ of $K\times K$. That is, we will show that the hyperk\"ahler quotient is homeomorphic to the GIT quotient $\Spec(\C[\mu_\C^{-1}(0)]^{H_\C})$ where $\mu_\C$ is the complex part of the hyperk\"ahler moment map. This requires first showing that $T^*G$ has a global K\"ahler potential that is proper and bounded from below, and we give a complete proof of that fact. With this theorem, we can study $\D(\g)$ and its orbit type partition algebraically and obtain the above results. The proofs about the stratification of $\D(\g)$ rely on the fact that we are quotienting by a torus and hence can use a weight decomposition.

\subsubsection{Organization of the paper} In \S\ref{vgklf6ut9n}, we give the necessary background on stratified spaces, singular hyperk\"ahler quotients, root subsystems, and the hyperk\"ahler space $T^*G$. We then state precisely the results of this paper in \S\ref{jwptxktt2u}. We prove the Kempf--Ness type theorem in \S\ref{udf7cgmv9f}, and prove all facts about the stratification of $\D(\g)$ in \S\ref{7isy6y5jpy}. Finally, we give examples in \S\ref{aejay646nw}.

\subsubsection{Acknowledgments} I thank Andrew Dancer, my PhD supervisor, for suggesting me the problem and for his guidance. I also thank Frances Kirwan, Brent Pym and Dan Ciubotaru for helpful discussions. I am also grateful to the referees for their careful reading and useful comments which helped improve the paper.

\section{Preliminaries}\label{vgklf6ut9n}

\subsection{Stratified spaces and singular quotients}\label{lm4qtiz8pq}

Stratified spaces are generalizations of manifolds that allow singularities. Roughly speaking, it is a space that can be decomposed into manifolds (of possibly different dimensions) which fit together nicely. There is no universally agreed definition of stratified spaces but, for the purpose of this paper, we use the following one (cf. \cite[Definition 1.1]{sja91} \cite[Definition 2.7.3]{dui12}).

\begin{definition}\label{0w8kdfpkuc}
Let $M$ be a topological space. A {\bf stratification} of $M$ is a partition $\cP$ of $M$ satisfying the following conditions.
\begin{itemize}
\item {\bf Manifold condition.} The elements of $\cP$ are topological manifolds in the subspace topology.
\item {\bf Local condition.} $\cP$ is locally finite and its elements are locally closed.
\item {\bf Frontier condition.} For all $S,T\in\cP$, we have  $S\cap \overline{T}\neq\emptyset$ $\Longrightarrow$ $S\s\overline{T}$.
\end{itemize}
The elements of $\cP$ are called the {\bf strata} and the pair $(M,\cP)$ a {\bf stratified space}.
\end{definition}

There is a natural relation on the partition $\cP$ of a stratified space $M$ given by $S\leq T$ if $S\s\overline{T}$. It follows from the local closedness of the strata that this relation is a partial order. We call $(\cP,\le)$ the {\bf stratification poset} and view it as an abstract partially ordered set (poset). In this framework, the frontier condition is equivalent to
\[
\overline{S}=\bigcup_{T\le S}T,\quad\text{for all }S\in\cP.
\]
Thus, the stratification poset $(\cP,\le)$ is a key structural information about $M$: it captures the precise way in which the strata fit together. One of the main goals of this paper is to describe this poset for the space $\D(\g)$ in terms of combinatorial data associated to $\g$.

We will use {\bf Hasse diagrams} to visualize the stratification posets. This is the diagram obtained by drawing a node for each $i\in \cP$ and an edge upward from node $i$ to node $j$ if $i<j$ but there is no $k$ such that $i<k<j$. In particular, the highest dimensional strata are on the top and lowest dimensional one on the bottom. For example, the following figure represents a stratification of the filled equilateral triangle and the corresponding Hasse diagram:

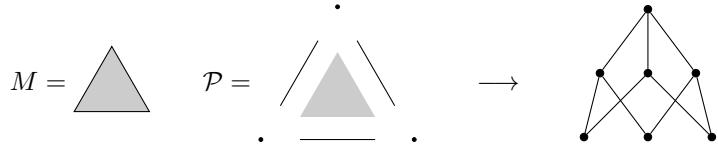
\[
M=
\begin{tikzpicture}[baseline={([yshift=-1ex]current bounding box.center)}]
\coordinate (G0) at (0, 0) {};
\coordinate (G1) at (1, 0) {};
\coordinate (G2) at (0.5, 0.866025) {};
\filldraw[draw=black, fill=black!20] (G0) -- (G1) -- (G2) -- cycle;
\end{tikzpicture}
\qquad
\cP=
\begin{tikzpicture}[baseline={([yshift=-1.5ex]current bounding box.center)}]
\coordinate (0) at (0, 0) {};
\coordinate (1) at (1, 0) {};
\coordinate (2) at (0.5, 0.866025) {};
\fill[black!20] (0) -- (1) -- (2) -- cycle;

\coordinate (3) at (-0.259808, 0.15) {};
\coordinate (4) at (0.240192, 1.01603) {};
\coordinate (5) at (0, -0.3) {};
\coordinate (6) at (1, -0.3) {};
\coordinate (7) at (1.25981, 0.15) {};
\coordinate (8) at (0.759808, 1.01603) {};
\draw (3) -- (4);
\draw (5) -- (6);
\draw (7) -- (8);

\node[draw, circle, fill, inner sep=0.5pt] (9) at (0.5, 1.46603) {};
\node[draw, circle, fill, inner sep=0.5pt] (10) at (-0.519614, -0.3) {};
\node[draw, circle, fill, inner sep=0.5pt] (11) at (1.51962, -0.3) {};
\end{tikzpicture}
\qquad
\too
\qquad
\begin{tikzpicture}[baseline={([yshift=-1.5ex]current bounding box.center)}, xscale=0.85, yscale=0.85]
\tikzset{dot/.style={draw, circle, fill, inner sep=1pt},}
\node[dot] (0) at (1, 2) {};
\node[dot] (1) at (0.25, 1) {};
\node[dot] (3) at (1, 1) {};
\node[dot] (2) at (1.75, 1) {};
\node[dot] (5) at (0, 0) {};
\node[dot] (6) at (1, 0) {};
\node[dot] (7) at (2, 0) {};
\draw (0) -- (1);
\draw (0) -- (2);
\draw (0) -- (3);
\draw (5) -- (1);
\draw (6) -- (2);
\draw (6) -- (1);
\draw (7) -- (2);
\draw (7) -- (3);
\draw (5) -- (3);
\end{tikzpicture}
\]
Many interesting examples of stratified spaces come from quotients by non-free actions. We review next three classes of examples.

\subsubsection{Smooth quotients} If a compact Lie group $K$ acts smoothly, but not necessarily freely, on a smooth manifold $M$, then the quotient $M/K$ has a natural stratification. It is described as follows. For each closed subgroup $H\s K$, let $(H)$ be the conjugacy class of $H$ in $K$. We say that $p\in M$ has {\bf orbit type} $(H)$ if its stabilizer subgroup $K_p$ is in $(H)$, and denote the set of points of orbit type $(H)$ by
\[
M_{(H)}:=\{p\in M:K_p\in(H)\}.
\]
The {\bf orbit type partition} $\cP=\{M_{(K_p)}/K:p\in M\}$ is a stratification of $M/G$ in the sense of Definition \ref{0w8kdfpkuc}, except for the fact that the strata may have connected components of different dimensions. In any case, we can refine the partition to get a genuine stratification. See \cite[\S2.7]{dui12} for details.

\subsubsection{K\"ahler quotients} Recall that if $M$ is a symplectic manifold, $K$ a compact Lie group acting freely on $M$ by preserving the symplectic form, and $\mu:M\to\fk^*$ a moment map for this action, then the Marsden--Weinstein reduction \cite{mar74}
\[
M\sll{}K:=\mu^{-1}(0)/K
\]
is a smooth symplectic manifold (we consider only quotients at level zero in this paper). Moreover, if $M$ is K\"ahler and $K$ preserves the K\"ahler structure, then $M\sll{}K$ is also K\"ahler \cite[Theorem 3.1]{hit87} and is called the \textbf{K\"ahler quotient} of $M$ by $K$. Sjamaar--Lerman \cite{sja91} generalized this construction to non-free actions by showing that $M\sll{}K$ has a natural stratification in which the strata are K\"ahler. The strata are $(M\sll{}K)_{(K_p)}:=\mu^{-1}(0)_{(K_p)}/K$ for $p\in \mu^{-1}(0)$ with the quotient metric and the K\"ahler form is characterized by the fact that its pullback to $\mu^{-1}(0)_{(K_p)}$ is the restriction of the K\"ahler form of $M$. (Again the strata may not be of pure dimension, but we can refine the partition.)

\subsubsection{Hyperk\"ahler quotients} Some of the results of Sjamaar--Lerman for singular K\"ahler quotients have analogues in hyperk\"ahler geometry. Let $M$ be a hyperk\"ahler manifold with complex structures $I_1,I_2,I_3$ and corresponding K\"ahler forms $\omega_1,\omega_2,\omega_3$. If $K$ is a compact Lie group acting on $M$ by preserving the hyperk\"ahler structure, then a {\bf hyperk\"ahler moment map} is a map $\mu=(\mu_1,\mu_2,\mu_3):M\to(\fk^*)^3$ such that $\mu_i$ is a moment map with respect to $\omega_i$. We call the triple $(M,K,\mu)$ a {\bf tri-Hamiltonian space}. A standard result of \cite[Theorem 3.2]{hit87} says that, if $K$ acts freely, the {\bf hyperk\"ahler quotient}
\[
M\slll{}K:=\mu^{-1}(0)/K
\]
is smooth and hyperk\"ahler in a canonical way.

More generally, Dancer--Swann \cite[Theorem 2.1]{dan97} proved that, for non-free $K$-actions, $M\slll{}K$ can be partitioned into smooth hyperk\"ahler manifolds. As for K\"ahler quotients, this is the {\bf orbit type partition}
\[
\cP=\{(M\slll{}K)_{(K_p)}:p\in \mu^{-1}(0)\}
\]
where $(M\slll{}K)_{(H)}:=\mu^{-1}(0)_{(H)}/K$. Thus, the manifold condition is satisfied and it follows from the results on smooth quotients that the local condition also holds.

However, it is not known if the frontier condition holds. The issue is that to prove this condition in the K\"ahler case, Sjamaar--Lerman use a local normal form for the moment map, but there is no known equivalent in the hyperk\"ahler setting.

The space $\D(\g)$ that we study in this paper is an example of a hyperk\"ahler quotient by a non-free action whose orbit type partition is a stratification.

\subsection{Root subsystems and regular subalgebras}\label{nj5e2b9c36}

The stratification poset of $\D(\g)$ is very rich and its description in terms of Lie theoretic data will occupy a good part of the paper. A crucial ingredient is the notion of root subsystems and regular subalgebras, which we review in this section. The results of this section can be found, for example, in \cite[Chapter 6]{on94}.

Let $\g$ be a complex semisimple Lie algebra with Cartan subalgebra $\ft$, $\Phi\s\ft^*$ the set of roots, and $\g=\ft\oplus\bigoplus_{\alpha\in\Phi}\g_\alpha$ the Cartan decomposition. A subset $\Psi$ of $\Phi$ is called a {\bf root subsystem} if
\begin{enumerate}
\item[(1)] $\alpha,\beta\in\Psi$ and $\alpha+\beta\in\Phi$ $\Longrightarrow$ $\alpha+\beta\in\Psi$, 
\item[(2)] $\alpha\in\Psi$ $\Longleftrightarrow$ $-\alpha\in\Psi$.
\end{enumerate}
Equivalently, $\Psi$ is a root subsystem if $(\Span_\Z\Psi)\cap\Phi=\Psi$. A root subsystem is always a root system itself. We use the notation $\Psi\le\Phi$ to say that $\Psi$ is a root subsystem of $\Phi$. This gives a {\it partial order} on the set of root subsystems.

Root subsystems are closely related to the notion of regular subalgebras. Recall that a subalgebra $\h$ of $\g$ is called {\bf regular} if there exists a Cartan subalgebra $\ft$ of $\g$ such that $[\ft,\h]\s\h$ (this notion was studied by Dynkin \cite{dyn72} who introduced this terminology). We denote by $\mathcal{C}_\g$ the set of conjugacy classes of regular semisimple subalgebras of $\g$. Since all Cartan subalgebras are conjugate, every element of $\mathcal{C}_\g$ has a representative which is regular with respect to a fixed Cartan subalgebra $\ft$.

\begin{proposition}
The set of semisimple subalgebras of $\g$ that are regular with respect to $\ft$ is in one-to-one correspondence with the set of root subsystems of $\Phi$. The correspondence associates to $\Psi\le\Phi$ the subalgebra
\begin{equation}\label{yepi9uvjbz}
\g_\Psi := \ft_\Psi\oplus\bigoplus_{\alpha\in\Psi}\g_\alpha,
\end{equation}
where $\ft_\Psi$ is the span of the coroots $h_\alpha$ for $\alpha\in\Psi$.
\end{proposition}

Recall that if $k$ is the number of simple factors of $\g$, there is an $(\R_{>0})^k$ family of non-degenerate symmetric invariant bilinear forms on $\g$ which are positive definite on the real span of the coroots. We call those bilinear forms {\bf admissible}. Equivalently, a bilinear form is admissible if its restriction to some (and hence all) compact real form(s) is negative definite. For example, the Killing form is admissible. 

\begin{proposition}\label{utfyudsumb}
Any admissible bilinear form on $\g$ remains admissible on $\g_\Psi$. Also, $\ft_\Psi$ is a Cartan subalgebra of $\g_\Psi$, \eqref{yepi9uvjbz} is the corresponding Cartan decomposition, and the map $\ft^*\to\ft_\Psi^*$ restricts to an isomorphism of abstract root systems from $\Span_\R\Psi$ to the root system of $\g_\Psi$ with respect to the above bilinear form.
\end{proposition}

Let $W_\g=W_\Phi=W(\g,\ft)$ be the Weyl group of $\g$ with respect to $\ft$. It acts on the set of root subsystems by $w\cdot\Psi:=\{w\cdot\alpha:\alpha\in\Psi\}$.

\begin{proposition}\label{cv6qokpmij}
Let $\Psi_1$ and $\Psi_2$ be two root subsystems. Then, $\g_{\Psi_1}=\g_{\Psi_2}$ if and only if there exists $w\in W_\Phi$ such that $w\cdot\Psi_1=\Psi_2$. Thus, the map $\Psi\mto\g_\Psi$ descends to a bijection $\{\textrm{root subsystems of }\Phi\}/W_\Phi\to\mathcal{C}_\g$.
\end{proposition}

There is also a natural partial order on $\mathcal{C}_\g$ induced by inclusion: we say that $[\h_1]\le[\h_2]$ if there exists an inner automorphism $\varphi$ of $\g$ such that $\varphi(\h_1)\s\h_2$.

Let $\h$ be a regular semisimple subalgebra of $\g$, say $\h=\g_\Psi$ for some $\Psi\s\Phi$. Then, the Weyl group $W_\h$ of $\h$ can be viewed as a subgroup of $W_\g$, namely the one generated by the simple reflections $s_\alpha$ for $\alpha\in\Psi$. In particular, the index $|W_\g:W_\h|:=|W_\g|/|W_\h|$ is a well-defined positive integer. We define the {\bf embedding number} of $\h$ in $\g$ to be
\begin{equation}\label{t9mcd8c0dw}
m_\g(\h)=|W_\g:W_\h||\{w\cdot \Psi:w\in W_\g\}|,
\end{equation}
where the second factor is the number of root subsystems in the $W_\g$-orbit of $\Psi$. The embedding number is thus a positive integer which depends on the particular way in which $\h$ embeds in $\g$. It depends only on the conjugacy class of $\h$ and hence descends to a map $m_\g:\mathcal{C}_\g\to\N$. Note that we always have $m_\g(\g)=1$ and $m_\g(0)=|W_\g|$. These numbers will be important for our study of the stratified space $\D(\g)$, as they will count the number of connected components of the strata.

\subsection{The hyperk\"ahler space $T^*G$}\label{4vlopz5r}

Let $G$ be a complex reductive group. Then, $T^*G$, being the cotangent bundle of a complex manifold, has a canonical holomorphic symplectic form. Kronheimer \cite{kro88} showed that there is a hyperk\"ahler structure on $T^*G$ compatible with this form. In other words, if $I_1,I_2,I_3$ are the three complex structures of this hyperk\"ahler structure and $\omega_1,\omega_2,\omega_3$ the associated K\"ahler forms, then $I_1$ is the natural complex structure on $T^*G$ and $\omega_2+i\omega_3$ the canonical holomorphic-symplectic form. This hyperk\"ahler structure is constructed by an infinite-dimensional version of the hyperk\"ahler quotient construction where the role of the moment map is played by a system of nonlinear ODEs called {\it Nahm's equations}. It requires a choice of a compact real form $K$ of $G$ and an invariant inner-product on $\fk=\Lie(K)$. However, since all compact real forms are conjugate, the hyperk\"ahler structure does not depend on $K$ in an essential way. But it does depend on the invariant inner-product on $\fk$, or equivalently, on an admissible bilinear form on $\g=\Lie(G)$.

Recall that for any hyperk\"ahler manifold $(M,g,I_1,I_2,I_3)$, if $(a_1,a_2,a_3)\in\R^3$ lies in the $2$-sphere, then the endomorphism $a_1I_1+a_2I_2+a_3I_3$ is a complex structure with respect to which $(M,g)$ is K\"ahler. Hence, $M$ has a {\bf 2-sphere of complex structures}. In the case of $T^*G$, there is an isometric $\SO(3)$-action which rotates these complex structures so, from the point of view of K\"ahler geometry, they are all equivalent.

The hyperk\"ahler space $T^*G$ also has a large group of symmetries: Dancer--Swann \cite{dan96} showed that the action of $K\times K$ on $G$ by left and right multiplications lifts to an action on $T^*G$ which preserves the hyperk\"ahler structure and has a canonical hyperk\"ahler moment map (which comes from evaluating solutions to Nahm's equations at the endpoints of the interval on which they are defined \cite[Lemma 2]{dan96}). Moreover, this action commutes with the $\SO(3)$-action. 

Let $H$ be a closed subgroup of $K\times K$ and $\mu$ the induced hyperk\"ahler moment map for the action of $H$ on $T^*G$, i.e.\ the composition of Dancer--Swann's hyperk\"ahler moment map $T^*G\to(\fk\times\fk)^*\otimes\R^3$ with the restriction map $(\fk\times\fk)^*\otimes\R^3\to\h^*\otimes\R^3$. Write $\mu_\C:=\mu_2+i\mu_3$ for the {\bf complex part} of $\mu$ and $\mu_\R:=\mu_1$ for the {\bf real part} of $\mu$. There is an algebraic description of $\mu_\C$ which is useful for computations. First, there is a complex algebraic isomorphism $T^*G\to G\times\g^*$ given by translating every cotangent space to the identity using left multiplication. On the latter space, the action of $K\times K$ is $(a,b)\cdot(g,\xi)=(agb^{-1},\Ad_b^*\xi)$ and this extends to an algebraic action of $G\times G$ which preserves the canonical holomorphic-symplectic form. Then, $\mu_\C$ is the composition of the map
\begin{equation}\label{g6ydwgujbz}
G\times\g^*\too\g^*\times\g^*,\quad (g,\xi)\mtoo(\Ad_g^*\xi,-\xi)
\end{equation}
with the restriction $\g^*\times\g^*\to\h^*_\C$ and is a holomorphic-symplectic moment map for the action of $H_\C\s G\times G$ on $T^*G$. See \cite[\S4]{dan96} and \cite[\S2]{bie97} for details.

\section{Statement of results}\label{jwptxktt2u}

In this section, we state precisely the main results of this paper. The proofs will be in the subsequent sections.

Let $G$ be a connected complex semisimple Lie group, $K$ a compact real form of $G$, and $H$ a closed subgroup of $K\times K$. There is a canonical hyperk\"ahler moment map for the action of $K\times K$ on $T^*G$ \cite[\S3]{dan96}, which is in fact unique since here $K$ is semisimple. Let $\mu = (\mu_\R, \mu_\C)$ be the induced hyperk\"ahler moment map for the $H$-action as in \S\ref{4vlopz5r}. By \S\ref{lm4qtiz8pq}, the hyperk\"ahler quotient $T^*G\slll{}H$ can be partitioned into smooth hyperk\"ahler manifolds, but it is not known if this is a stratification in the sense of Definition \ref{0w8kdfpkuc}. 

We first give an algebraic description of the topological space $T^*G\slll{}H$ and its orbit type partition which works for all $H$ and needs no reference to the hyperk\"ahler setting. Recall that the set of {\bf polystable points} $\mu_\C^{-1}(0)\ps$ is the set of points in $\mu_\C^{-1}(0)$ whose $H_\C$-orbit is closed. The inclusion $\mu_\C^{-1}(0)\ps\hookrightarrow\mu_\C^{-1}(0)$ descends to a bijection from $\mu_\C^{-1}(0)\ps/H_\C$ to the GIT quotient $\Spec(\C[\mu_\C^{-1}(0)]^{H_\C})$ \cite[\S I.1]{lun73} and, moreover, this map is a homeomorphism in the Euclidean topology \cite[\S2.1]{lun76} \cite[Theorem 2.1]{flo14}. Thus, the $H_\C$-orbit type partition on $\mu_\C^{-1}(0)\ps/H_\C$ induces a partition on $\Spec(\C[\mu_\C^{-1}(0)]^{H_\C})$ sometimes called the algebraic or Luna stratification \cite[\S III.2]{lun73} \cite[\S6.9]{pop94}.

\begin{theorem}\label{9up121d4yh}
We have $\mu^{-1}(0)\s\mu_\C^{-1}(0)\ps$ and this inclusion induces a homeomorphism
\[
T^*G\slll{}H\too \mu_\C^{-1}(0)\ps/H_\C=\Spec(\C[\mu_\C^{-1}(0)]^{H_\C}).
\]
Moreover, the $H$-orbit type partition of $T^*G\slll{}H$ coincides with the $H_\C$-orbit type partition of $\mu_\C^{-1}(0)\ps/H_\C$ and the homeomorphism restricts to biholomorphisms on the strata.
\end{theorem}

The proof follows from a slight generalization of the classic Kempf--Ness theorem to varieties with global K\"ahler potentials. The main step is to show that every point is analytically semistable, i.e.\ that the closure of the $H_\C$-orbit of every point in $\mu_\C^{-1}(0)$ intersects $\mu_\R^{-1}(0)$. One has to go to the gauge-theoretic point of view using Nahm's equations and do some analysis since $\mu_\R$ has no explicit algebraic description as a map on $T^*G$ (the issue is that the passage from $T^*G$ to Nahm's equations relies on an existence result in analysis and hence is not explicit). Note that the theorem implies that the pieces in the orbit type partition of $\mu_\C^{-1}(0)\ps/H_\C$ are smooth, which is not immediately obvious since $\mu_\C^{-1}(0)$ might be singular. In addition, the theorem says that with respect to any element in the 2-sphere of complex structures, $T^*G\slll{}H$ has the structure of a complex affine variety.

We are mostly interested in the special case
\[
\D(G):=T^*G\slll{}(T_K\times T_K)
\]
where $T_K$ is a maximal torus in $K$. Note that $\D(G)$ does not depend on the choice of maximal torus since they are all conjugate. On the other hand, the hyperk\"ahler structure does depend on the choice of admissible bilinear form on $\g$. A crucial ingredient in the description of $\D(G)$ is that, in fact, it depends {\it only} on $\g$ and its bilinear form. More precisely:

\begin{theorem}\label{qaqfm3dpar}
Let $\tilde{G}$ be the universal cover of $G$. Then, there is a homeomorphism $\D(\tilde{G})\to\D(G)$ which preserves the orbit type partition and restricts to hyperk\"ahler isomorphisms on the strata.
\end{theorem}

This justifies the notation $\D(\g)$ instead of $\D(G)$. The main point about taking $H=\allowbreak T_K \times T_K$ is that the GIT quotient $\Spec(\C[\mu_\C^{-1}(0)]^{H_\C})$ can be studied with a weight decomposition. This enables us to prove:

\begin{theorem}\label{2rzpnjamfd}
The orbit type partition of $\D(\g)$ is a stratification.
\end{theorem}

More precisely, the proof relies on an explicit description of the orbit type partition. Let $\cO\s\g$ be a regular semisimple adjoint orbit and consider the action of $T:=(T_K)_\C$ on the cotangent bundle $T^*\cO\s\g\times\g^*$ by $t\cdot(X,\eta)=(\Ad_tX,\Ad_t^*\eta)$. Let $\ft^\circ$ be the annihilator of $\ft$ in $\g^*$. Using Theorem \ref{9up121d4yh}, we get the following description, which is well suited for computations.

\begin{proposition}\label{0ae71l8pfy}
The space $\D(\g)$ is homeomorphic to $(T^*\cO\cap(\g\times\ft^\circ))\ps/T$. Moreover, the $(T_K\times T_K)$-orbit type partition of $\D(\g)$ coincides with the $T$-orbit type partition of $(T^*\cO\cap(\g\times\ft^\circ))\ps/T$.
\end{proposition}

In particular, the orbit type strata of $\D(\g)$ are of the form $(T^*\cO\cap(\g\times\ft^\circ))\ps_{Z}/T$ for some closed subgroups $Z$ of $T$ (we use the subscript $Z$ instead of $(Z)$ since $T$ is abelian). Let $\Phi$ be the root system of $\g$ and view it as a subset of $\Hom(T,\C^*)$. For every root subsystem $\Psi\le\Phi$, let

\begin{equation}\label{6hygdax5ph}
\begin{aligned}
Z_\Psi &:= \{t\in T:\alpha(t)=1\text{ for all }\alpha\in\Psi\} \\
\D(\g)_\Psi &:= (T^*\cO\cap(\g\times\ft^\circ))\ps_{Z_\Psi}/T.
\end{aligned}
\end{equation}

\begin{theorem}\label{zt3pxhz7ur}
The map
\[
\{\text{root subsystems of }\Phi\}\too\{\text{orbit type strata of }\D(\g)\},\quad\Psi\mtoo\D(\g)_\Psi
\]
is an isomorphism of posets.
\end{theorem}

In particular, there is a top stratum $\D(\g)_{\mathrm{top}}:=\D(\g)_\Phi$ corresponding to the root system $\Phi$ itself and a bottom stratum $\D(\g)_{\mathrm{bottom}}:=\D(\g)_\emptyset$ corresponding to the trivial root subsystem $\emptyset$. Recall from \S\ref{nj5e2b9c36} that any admissible bilinear form on $\g$ remains admissible on $\g_\Psi$. Hence, it induces a natural hyperk\"ahler structure on $\D(\g_\Psi)_{\mathrm{top}}$. Let $W$ be the Weyl group of $\g$ and $W_\Psi$ the Weyl group of $\g_\Psi$.

\begin{theorem}\label{6bfrjygjvw}\
\begin{itemize}
\item[$(1)$] $\D(\g)_{\mathrm{top}}$ is a connected open dense subset of real dimension $4(\dim\g-2\rk\g)$.
\item[$(2)$] $\D(\g)_{\mathrm{bottom}}$ is a finite set of $|W|$ elements. 
\item[$(3)$] For all $\Psi\le\Phi$, the stratum $\D(\g)_\Psi$ is isomorphic as a hyperk\"ahler manifold to a disjoint union of $|W:W_\Psi|$ copies of $\D(\g_\Psi)_{\mathrm{top}}$.
\end{itemize}
\end{theorem}

A stratification can always be refined arbitrarily by taking submanifolds of the strata. It is thus desirable to get a stratification as coarse as possible. We describe next one way of coarsening the orbit type stratification of $\D(\g)$. Recall from \S\ref{nj5e2b9c36} that $\mathcal{C}_\g$ is the set of conjugacy classes of regular semisimple subalgebras of $\g$ and $m_\g:\mathcal{C}_\g\to\N$ the map of embedding numbers \eqref{t9mcd8c0dw}. For each $[\h]\in\mathcal{C}_\g$, let
\[
\D(\g)_{[\h]} := \bigcup_{[\g_\Psi]=[\h]}\D(\g)_\Psi.
\]

\begin{theorem}\label{xvislcts8s}
The partition $\cP=\{\D(\g)_{[\h]}:[\h]\in\mathcal{C}_\g\}$ is a stratification of $\D(\g)$ and the map $\mathcal{C}_\g\to\cP$, $[\h]\mto\D(\g)_{[\h]}$ an isomorphism of posets. Moreover, $\D(\g)_{[\g]}=\D(\g)_{\mathrm{top}}$, $\D(\g)_{[0]}=\D(\g)_{\mathrm{bottom}}$ and $\D(\g)_{[\h]}$ is isomorphic as a hyperk\"ahler manifold to a disjoint union of $m_\g(\h)$ copies of $\D(\h)_{\mathrm{top}}$. 
\end{theorem}

We will give examples of this stratification in \S\ref{aejay646nw}.

\section{A Kempf--Ness theorem for $T^*G\slll{}H$}\label{udf7cgmv9f}

The purpose of this section is to prove Theorem \ref{9up121d4yh} relating the hyperk\"ahler quotient $T^*G\slll{}H$ with the GIT quotient $\Spec(\C[\mu_\C^{-1}(0)]^{H_\C})$.

\subsection{K\"ahler quotients with global K\"ahler potentials}

Let $G$ be a complex reductive group with maximal compact subgroup $K$ and $M$ a smooth complex affine variety on which $G$ acts algebraically.

Recall that the Kempf--Ness theorem \cite{kem79} \cite[\S6.12]{pop94} \cite[Corollary 4.7]{sch89} 	states that if $M\s\C^n$ is endowed with the standard K\"ahler structure and $G$ acts linearly on $\C^n$ and $K$ by isometries, then there is a canonical moment map $\mu$ for the action of $K$ on $M$ such that the K\"ahler quotient $\mu^{-1}(0)/K$ is homeomorphic to the GIT quotient $\Spec(\C[M]^G)$. However, in our case, the K\"ahler structures on $T^*G$ cannot be obtained from embeddings in $\C^n$, so we need a slight generalization of this theorem. 

More precisely, suppose that $M$ has a K\"ahler structure (compatible with its natural complex structure) which is induced by a global K\"ahler potential $f\in C^\infty(M,\R)$ that is $K$-invariant, proper and bounded from below. For example, if $M$ is isometrically embedded in $\C^n$, $f$ can be the norm function. Let $I$ be the complex structure of $M$ and $\omega=2i\partial\bar{\partial}f$ the K\"ahler form. For each $X\in\fk$, let $X^\#$ be the vector field on $M$ induced by the $K$-action and define
\[
\mu:M\too\fk^*,\quad \mu(p)(X)=df(IX^\#_p)
\]
for all $p\in M$ and $X\in\fk$.

\begin{proposition}\label{11c5cigbp4}
The map $\mu$ is a moment map for the action of $K$ on $M$.
\end{proposition}

\begin{proof}
We have $\omega=-d\alpha$ where $\alpha=i(\partial-\bar{\partial})f$. Moreover, $df(IX^\#)=\alpha(X^\#)$ so $i_{X^\#}\alpha=\mu^X$ where $\mu^X(p):=\mu(p)(X)$. Now, the Lie derivative $\mathcal{L}_{X^\#}$ commutes with $\partial$ and $\bar{\partial}$ since $K$ acts by biholomorphisms, so $\mathcal{L}_{X^\#}\alpha=i(\partial-\bar{\partial})\mathcal{L}_{X^\#}f=0$, by $K$-invariance of $f$. Hence, $i_{X^\#}\omega=d(i_{X^\#}\alpha)=d\mu^X$. To prove $K$-equivariance of $\mu$, define, for all $k\in K$, the map $\psi_k:M\to M$, $p\mto k\cdot p$. Then, $df\circ I\circ d\psi_k=d(f\circ\psi_k)\circ I=df\circ I$ and $X^\#_{k\cdot p}=d\psi_k((\Ad_{k^{-1}}X)^\#_p)$, so
\[
\mu(k\cdot p)(X) = df(I(d\psi_k((\Ad_{k^{-1}}X)^\#_p)))=df(I(\Ad_{k^{-1}}X)^\#_p)=\mu(p)(\Ad_{k^{-1}}X).
\]
Hence, $\mu(k\cdot p)=\Ad_k^*\mu(p)$ for all $k\in K$ and $p\in M$.
\end{proof}

\begin{proposition}\label{p5xq2g27v0}
We have $\mu^{-1}(0)\s M\ps$ and this inclusion induces a homeomorphism $\mu^{-1}(0)/K\to M\ps/G=\Spec(\C[M]^G)$. Moreover, the $K$-orbit type partition of $\mu^{-1}(0)/K$ coincides with the $G$-orbit type partition of $M\ps/G$ and the homeomorphism restricts to biholomorphisms on the strata.
\end{proposition}

\begin{proof}
Let $M\ass$ be the set of $p\in M$ such that the closure of $G\cdot p$ intersects $\mu^{-1}(0)$ and let $M\aps$ be the set of $p\in M\ass$ such that $G\cdot p$ is closed in $M\ass$ (`$\textrm{a-ss}$' stands for \textit{analytically semistable} and `$\textrm{a-ps}$' stands for \textit{analytically polystable}). Then, Sjamaar showed that $\mu^{-1}(0)\s M\aps$ \cite[Proposition 2.4]{sja95} and that this inclusion descends to a homeomorphism $\mu^{-1}(0)/K\to M\aps/G$ \cite[Remark 2.6]{sja95} (see also \cite{hei94}) with the desired properties regarding orbit type partitions \cite[Theorem 2.10]{sja95}. Thus, it suffices to show that $M\ass=M$, so that $M\aps=M\ps$. This is proved by adapting the standard proof of the Kempf--Ness theorem, replacing the norm function by $f$. Let $p_0\in M$. Since $G$ acts algebraically, there is a closed orbit $G\cdot q_0$ in $\overline{G\cdot p_0}$. We claim that $G\cdot q_0$ intersects $\mu^{-1}(0)$. Since $f$ is proper and bounded from below, $f(G\cdot q_0)$ attains a minimum $f(g_0\cdot q_0)$ for some $g_0\in G$. Thus, by defining $F_p:G\to\R$, $F_p(g)=f(g\cdot p)$, we get $(dF_{q_0})_{g_0}=0$. Moreover, if $R_g:G\to G$, $R(a)=ag$ then $F_{g\cdot p}=F_p\circ R_g$, so $(dF_{g_0\cdot q_0})_1=0$. Then, for all $X\in\fk$ we have $(dF_p)_1(iX)=df(IX^\#_p)=\mu(p)(X)$, so in particular $\mu(g_0\cdot q_0)(X)=(dF_{g_0\cdot q_0})_1(iX)=0$. Hence, $\mu(g_0\cdot q_0)=0$, so $\overline{G\cdot p_0}$ intersects $\mu^{-1}(0)$.
\end{proof}

\subsection{Application to $T^*G\slll{}H$}

We want to apply Proposition \ref{p5xq2g27v0} to a certain K\"ahler potential on $T^*G$. This potential is obtained from {\bf Nahm's equations}, i.e.\ the system of ordinary differential equations
\[
\dot{T}_1 = [T_2,T_3],\quad \dot{T}_2 = [T_3,T_1],\quad\dot{T}_3 = [T_1,T_2],
\]
where the $T_i$'s take values in $\fk$. 

For each $X\in\fk^3$, let $T^X$ be the maximally extended solution to Nahm's equations with initial condition $T^X(0)=X$. Let $W$ be the open set of $X\in\fk^3$ such that $T^X$ is defined at least on $[0,1]$. Then, Dancer--Swann \cite[\S3]{dan96} showed that $W$ is star-shaped about $0$, and $K\times W$ is diffeomorphic to $T^*G$. Moreover, they observed that the function
\[
F_{12}:K\times W\too\R,\quad F_{12}(k,X) = \frac{1}{2}\int_0^1\|T_1^X(t)\|^2+\|T_2^X(t)\|^2\,dt
\]
is a moment map with respect to $\omega_3$ for a $\mathrm{U}(1)$-action fixing $I_3$ while rotating $I_1$ and $I_2$, so, by \cite[\S3(E)]{hit87}, it is a K\"ahler potential for $(I_1,\omega_1)$ and $(I_2,\omega_2)$. Similarly, the function $F_{13}$ defined by $F_{13}(k,X)=\frac{1}{2}\int_0^1\|T_1^X(t)\|^2+\|T_3^X(t)\|^2\,dt$ is a K\"ahler potential for $(I_1,\omega_1)$ and $(I_3,\omega_3)$. In particular, both $F_{12}$ and $F_{13}$ are K\"ahler potentials for $(I_1,\omega_1)$. Neither $F_{12}$ nor $F_{13}$ is proper (we have, e.g., $F_{12}^{-1}(0)\cong K\times\fk$), but we will show that their average $F:=\frac{1}{2}(F_{12}+F_{13})$ is proper. That is, we will focus on the K\"ahler potential for $(I_1,\omega_1)$ on $T^*G\cong K\times W$ defined by
\begin{equation}\label{8bodvlu0lq}
F:K\times W\too\R,\quad F(k,X)=\frac{1}{4}\int_0^12\|T_1^X(t)\|^2+\|T_2^X(t)\|^2+\|T_3^X(t)\|^2\,dt.
\end{equation}
The diffeomorphism $T^*G\cong K\times W$ intertwines the standard $K\times K$-action on $T^*G$ with the action $(a,b)\cdot(k,X)=(a k b^{-1},\Ad_{b}X)$ on $K\times W$, so $F$ is $K\times K$-invariant and bounded from below. We prove that $F$ is proper by the following three lemmas.

\begin{lemma}\label{kcltzmm4tx}
Let $T$ be a solution to Nahm's equations on an interval $I$. Then, $\|T_i\|^2$ is convex on $I$ for all $i$.
\end{lemma}

\begin{proof}
We have $\frac{d}{dt}\|T_1\|^2=2\ip{T_1,\dot{T_1}}=2\ip{T_1,[T_2,T_3]}$ so
\begin{align*}
\frac{d^2}{dt^2}\|T_1\|^2 &= 2\ip{\dot{T}_1,[T_2,T_3]}+2\ip{T_1,[\dot{T}_2,T_3]+[T_2,\dot{T}_3]} \\
&= 2\ip{\dot{T}_1,[T_2,T_3]}+2\ip{[T_3,T_1],\dot{T}_2}+2\ip{[T_1,T_2],\dot{T}_3]} \\
&= 2\|[T_2,T_3]\|^2+2\|[T_3,T_1]\|^2+2\|[T_1,T_2]\|^2.
\end{align*}
Similar statements hold for $\|T_2\|^2$ and $\|T_3\|^2$.
\end{proof}

\begin{lemma}\label{sfhpewpbid}
Let $T$ be a solution to Nahm's equations on $[0,s)$. If $\int_0^{s}\|T_i\|^2dt<\infty$ for some $i$, then $T$ can be extended past $s$.
\end{lemma}

\begin{proof}
Let $\Phi:\fk^3\to\fk^3$, $\Phi(X_1,X_2,X_3)=([X_2,X_3],[X_3,X_1],[X_1,X_2])$ so that Nahm's equations become $\dot{T}(t)=\Phi(T(t))$. We want to show that there exists $t_0<s$ such that the unique solution to the initial value problem $\dot{S}(t)=\Phi(S(t))$, $S(t_0)=T(t_0)$ exists on $[t_0,t_0+\e)$ for some $\e>s-t_0$. From the existence and uniqueness theorem for first-order systems of ODEs, $\e$ can be taken to be $b/M$ where $b>0$ and
\[
M=\sup\{\|\Phi(X)\|:\|X-T(t_0)\|\leq b\};
\]
see \cite[Ch.\ 1 (Theorem 2.3 and 5th paragraph of p.\ 19)]{cod55}. Note that $\Phi$ is a homogeneous polynomial of degree $2$, so there exists $C>0$ such that $\|\Phi(X)\|\le C\|X\|^2$ for all $X\in\fk^3$. Hence, $M\leq C(\|T(t_0)\|+b)^2$ and it suffices to show that there exists $t_0<s$ and $b>0$ such that
\begin{equation}\label{m7ky6qszga}
\frac{b}{C(\|T(t_0)\|+b)^2}>s-t_0.
\end{equation}
We have $\frac{d}{dt}\|T_1\|^2=2\ip{T_1,[T_2,T_3]}=2\ip{[T_1,T_2],T_3}=\frac{d}{dt}\|T_3\|^2$ and similarly $\frac{d}{dt}\|T_2\|^2\allowbreak=\frac{d}{dt}\|T_1\|^2$, so the maps $\|T_i\|^2$ differ by constants. Hence, since $\int_0^{s}\|T_i\|^2dt<\infty$ for some $i$, we also have $\int_0^{s}\|T\|^2dt<\infty$. Moreover, $\|T\|^2$ is convex (Lemma \ref{kcltzmm4tx}) so the fact that $\int_0^{s}\|T\|^2dt<\infty$ implies that there exists $\delta>0$ such that $\|T(t)\|^2<\frac{1}{s-t}$ for all $t\in(s-\delta,s)$. Let $b>C$. By taking $\delta$ small enough, we may assume that $C(1+b\sqrt{\delta})^2<b$. Then, for all $t_0\in(s-\delta,s)$ we have
\[
C(s-t_0)(\|T(t_0)\|+b)^2<C(s-t_0)\left(\frac{1}{\sqrt{s-t_0}}+b\right)^2=C\left(1+b\sqrt{s-t_0}\right)^2<b
\]
and hence \eqref{m7ky6qszga} follows.
\end{proof}

\begin{lemma}
The map $F$ given by \eqref{8bodvlu0lq} is proper.
\end{lemma}

\begin{proof}
Since $F$ is independent of $K$, which is compact, it suffices to show that
\[
f:W\too\R,\quad 
X\mtoo\frac{1}{4}\int_0^12\|T_1^X(t)\|^2+\|T_2^X(t)\|^2+\|T_3^X(t)\|^2\,dt.
\]
is proper, i.e.\ that $f^{-1}([0,A])$ is compact for all $A>0$. We first show that $f^{-1}([0,A])$ is bounded, and then that it is closed in $\fk^3$ (which does not follow from continuity since the domain $W$ of $f$ is open in $\fk^3$).

Note that if $X\in W$ and $tX\in W$ for some $t>0$ then $T^{tX}(u)=tT^X(tu)$ and hence
\begin{equation}\label{h9qbuyvh03}
f(tX)=\frac{t}{4} \int_0^t2\|T^X_1(u)\|^2+\|T_2^X(u)\|^2+\|T_3^X(u)\|^2\,du.
\end{equation}
In particular, if $t\ge 1$ then $f(tX)\ge tf(X)$ and if $0<t\le1$ then $f(tX)\le tf(X)$. This will be used repeatedly in what follows.

Let $\e>0$ be such that $S_\e:=\{X\in\fk^3:\|X\|=\e\}\s W$. Then, $f(S_\e)$ is compact and hence attains a minimum $M\ge0$. We have $f(X)=0$ if and only if $X=0$, so in fact $M>0$. Let $Y\in\fk^3$ be such that $\|Y\|> \e$ and write $Y=tX$ where $X\in S_\e$ and $t>1$. Then, $f(Y)=f(tX)\ge t f(X)=tM$ and since $\|Y\|=t\e$, we get $\|Y\|\leq \frac{\e}{M}f(Y)$. Hence, $f^{-1}([0,A])$ is bounded for all $A>0$.

To show that $f^{-1}([0,A])$ is closed in $\fk^3$, let $X_n\in W$ be a sequence that converges to some $X\in\fk^3$ and satisfies $f(X_n)\leq A$ for all $n$. Suppose, for contradiction, that $X\notin W$. Since $W$ is star-shaped about $0$, we must have $\{t>0:tX\in W\}=(0,s)$ for some $0<s\le1$. Moreover, $f(tX)\to\infty$ as $t\to s$, since otherwise $T^{X}$ is defined on $[0,s]$ by \eqref{h9qbuyvh03} and Lemma \ref{sfhpewpbid}, so $T^{sX}(u)=sT^X(su)$ is defined for all $u\in[0,1]$ and hence $sX\in W$. Thus, there exists $t<s$ such that $f(tX)>A$. But $tX\in W$, so there exists $r>0$ such that the open ball $B_r(tX)$ is contained in $W$ and $f(Y)>A$ for all $Y\in B_r(tX)$. Let $n$ be such that $X_n\in B_r(X)$. Then, $|tX_n-tX|<tr<r$ so $f(tX_n)>A$. But since $t<1$ we have $f(tX_n)\leq tf(X_n)\le f(X_n)$, so $f(X_n)>A$, a contradiction. Thus, $X\in W$ and hence $f^{-1}([0,A])$ is closed in $\fk^3$.
\end{proof}

By the isometric $\SO(3)$-action rotating the complex structures, we have proved:

\begin{proposition}\label{v99uf9qhe1}
Let $I$ be any element in the 2-sphere of complex structures on $T^*G$. Then, there is a global K\"ahler potential for $I$ which is $K\times K$-invariant, proper, and bounded from below.
\end{proposition}

We can now prove the relation between $T^*G\slll{}H$ and $\Spec(\C[\mu_\C^{-1}(0)]^{H_\C})$:

\begin{proof}[Proof of Theorem \ref{9up121d4yh}]
Denote by $\mu_\R^H$ and $\mu_\R^{K\times K}$ the moment maps for $H$ and $K\times K$ on $T^*G$ that we are getting from the K\"ahler potential $F$ as in Proposition \ref{11c5cigbp4}. Then, $\mu_\R^H$ is equal to the composition of $\mu_\R^{K\times K}$ with the projection $\fk^*\times\fk^*\to\h^*$. Moreover, since $G$ is semisimple, so is $K\times K$, and hence moment maps for $K\times K$ are unique. Thus, $\mu_\R^{H}$ must be the same moment map as the one considered in Theorem \ref{9up121d4yh}. Since $F$ is proper, we can apply Proposition \ref{p5xq2g27v0} to the K\"ahler quotient $\mu_\R^{-1}(0)/H$. The hyperk\"ahler quotient $T^*G\slll{}H$ is obtained by taking the subspace of $\mu_\R^{-1}(0)/H$ given by $\mu_\C^{-1}(0)\cap\mu_\R^{-1}(0)/H$, so we get the theorem.
\end{proof}

\section{Stratification of $\D(\g)$ into hyperk\"ahler manifolds}\label{7isy6y5jpy}

\subsection{$\D(G)$ depends only on the Lie algebra of $G$}

In this section we prove Theorem \ref{qaqfm3dpar} which says that $\D(G):=T^*G\slll{}(T_K\times T_K)$ depends only on the Lie algebra $\g$ of $G$. This will be used in \S\ref{g3db9khyew} to identify the strata of $\D(\g):=\D(G)$ as the top stratum of spaces $\D(\h)$ for some semisimple subalgebras $\h\s\g$ (Theorem \ref{6bfrjygjvw}).

The first step is the explicit algebraic expression for $\D(G)$ claimed in Proposition \ref{0ae71l8pfy}. By Theorem \ref{9up121d4yh} we have $\D(G)=\mu_\C^{-1}(0)\ps/(T\times T)$ and from \eqref{g6ydwgujbz},
\[
\mu_\C^{-1}(0)=\{(g,\xi)\in G\times\ft^\circ:\Ad_g^*\xi\in\ft^\circ\}.
\]
Let $\cO$ be a regular semisimple orbit in $\g$, say $\cO=G\cdot\tau$ for some $\tau\in\ft\reg$. Then, the map $G\to\cO$, $g\mto\Ad_g\tau$ is a principal $T$-bundle and the associated vector bundle $G\times_T\ft^\circ$ (using the coadjoint action on $\ft^\circ$) is isomorphic to the cotangent bundle $T^*\cO$ via $(g,\xi)\mto(\Ad_g\tau,\Ad_g^*\xi)$ (see e.g.\ \cite[Lemma 1.4.9]{chr09}). Thus, we have
\[
\D(G)=\{(g,\xi)\in G\times_T\ft^\circ:\Ad_g^*\xi\in\ft^\circ\}\ps/T=(T^*\cO\cap(\g\times\ft^\circ))\ps/T,
\]
where $T$ acts on $T^*\cO\s\g\times\g^*$ by the adjoint and coadjoint actions. Moreover, the action of $T$ on $G\times\ft^\circ$ is free so the $(T\times T)$-orbit type partition of $\D(G)$ coincides with the $T$-orbit type partition of $(T^*\cO\cap(\g\times\ft^\circ))\ps/T$. Another way of viewing this is to use reduction in stages \cite[\S4]{sja91} to get $\D(G)=(T^*G\slll{}(1\times T_K))\slll{}(T_K\times 1)$ and note that by Theorem \ref{9up121d4yh}, $T^*G\slll{}(1\times T_K)=G\times_T\ft^\circ=T^*\cO$.

We already see that $\D(G)$ depends only on $\g$ as an algebraic variety since $T^*\cO\cap(\g\times\ft^\circ)$ depends only on $\g$ and $T$ acts by the adjoint representation. Moreover, we have proved Proposition \ref{0ae71l8pfy}. But it remains to show that the hyperk\"ahler structures on the strata are also independent of coverings.

If a group $K$ acts on a set $M$ and $A\s K$, let $M_A:=\{p\in M:K_p=A\}$. The following simple observation will be useful here and in \S\ref{g3db9khyew}. 

\begin{lemma}\label{zkpb6ssaaq}
Let $(M,K,\mu)$ and $(N,L,\nu)$ be tri-Hamiltonian spaces, $f:K\to L$ a Lie group homomorphism and $F:M\to N$ an $f$-equivariant hyperk\"ahler map (isometric, tri-holomorphic and tri-symplectomorphic) which sends $\mu^{-1}(0)$ to $\nu^{-1}(0)$. If $K$ and $L$ act freely, then $F$ descends to a hyperk\"ahler map $\bar{F}:M\slll{}K\to N\slll{}L$. More generally, if $F(M_A)\s N_B$ for some $A\s K$ and $B\s L$, then the restriction $\bar{F}:(M\slll{}K)_{(A)}\to (N\slll{}L)_{(B)}$ is a hyperk\"ahler map.
\end{lemma}

\begin{proof}
Suppose first that $K$ and $L$ act freely. Since the three symplectic forms on $M\slll{}K$ are uniquely characterized by the fact that their pullback to $\mu^{-1}(0)$ are the restrictions of the symplectic forms of $M$ \cite[proof of Theorem 3.2]{hit87} (and similarly for $N\slll{}L$), it is immediate that $\bar{F}$ is a symplectomorphism for each of them. A hyperk\"ahler structure is uniquely determined by its metric and three symplectic forms, so it suffices to show that $\bar{F}$ is an isometry. Let $g$ be the metric on $M$ and $\pi:\mu^{-1}(0)\to\mu^{-1}(0)/K$ the quotient map. Then, $\pi$ is a principal $K$-bundle and the metric $\bar{g}$ on $\mu^{-1}(0)/K$ is uniquely determined by $\pi^*\bar{g}(u,v)=g(u,v)$ for all {\it horizontal} tangent vectors $u,v$ on $\mu^{-1}(0)$. A similar statement holds for the metric $\bar{h}$ on $N\slll{}L$. Since $F$ is an isometry and is $f$-equivariant, it follows that $\bar{F}$ maps horizontal vectors of $\mu^{-1}(0)$ to horizontal vectors of $\nu^{-1}(0)$. Let $\rho:\nu^{-1}(0)\to\nu^{-1}(0)/L$ be the quotient map. Then, for all horizontal vectors $u,v$ on $\mu^{-1}(0)$, we have $\pi^*(\bar{F}^*\bar{h})(u,v)=F^*\rho^*\bar{h}(u,v)=\rho^*\bar{h}(dF(u),dF(v))=\allowbreak h(dF(u),dF(v))=g(u,v)$, so $\bar{F}^*\bar{h}=\bar{g}$.

Now suppose that $F(M_A)\s N_B$ for some $A\s K$ and $B\s L$. Recall from \cite[\S2]{dan97} that $M_A$ is a hyperk\"ahler submanifold of $M$ on which $N_G(A)/A$ acts freely and $(M_A,N_G(A)/A,\mu|_{M_A})$ is a tri-Hamiltonian space. The hyperk\"ahler structure on $(M\slll{}K)_{(A)}=\mu^{-1}(0)_{(A)}/K$ is given by identifying this space with $M_A\slll{}(N_G(A)/A)$. Thus, we may repeat the argument above with $M$ and $N$ replaced by $M_A$ and $N_B$.
\end{proof}

We can now prove Theorem \ref{qaqfm3dpar}. The idea is to use Lemma \ref{zkpb6ssaaq} on the covering $T^*\tilde{G}\to\allowbreak T^*G$ and show that the map $\D(\tilde{G})\to\D(G)$ that it induces is a homeomorphism preserving the stratifications.

\begin{proof}[Proof of Theorem \ref{qaqfm3dpar}]
Let $\pi:\tilde{G}\to G$ be the covering map and $\tilde{K}$, $K$ maximal compact subgroups of $\tilde{G}$ and $G$ respectively with the same Lie algebra $\fk\s\g$. Then, the map $F : \tilde{G}\times\g \to G\times\g$, $(g,X)\mto(\pi(g),X)$ is hyperk\"ahler since it descends from the identity map on the space of solutions to Nahm's equations on $\fk$ (see the description of Kronheimer \cite{kro88}). Let $T_{\tilde{K}}$ and $T_K$ be maximal tori in $\tilde{K}$ and $K$ respectively, with the same Lie algebra $\ft_\fk\s\fk$. Then, $F$ is equivariant with respect to the covering $T_{\tilde{K}}\times T_{\tilde{K}}\to T_K\times T_K$. Moreover, if $\tilde{\mu}$ and $\mu$ are the moment maps for the actions of $T_{\tilde{K}}\times T_{\tilde{K}}$ and $T_K\times T_K$ respectively, then $\mu\circ F=\tilde{\mu}$ so $F$ maps $\tilde{\mu}^{-1}(0)$ to $\mu^{-1}(0)$. Thus, Lemma \ref{zkpb6ssaaq} ensures that $F$ descends to a continuous map $\bar{F}:\D(\tilde{G})\to\D(G)$ which is hyperk\"ahler on the strata that it preserves. But, the diagram
\[
\begin{tikzcd}
\D(\tilde{G})\arrow{r}{\bar{F}}\arrow{d} & \D(G)\arrow{d} \\
(T^*\cO\cap(\g\times\ft^\circ))\ps/\tilde{T}\arrow{r} & (T^*\cO\cap(\g\times\ft^\circ))\ps/T
\end{tikzcd}
\]
commutes, where the vertical and bottom maps are homeomorphisms preserving the orbit type partition, so $\bar{F}$ is also a homeomorphism preserving the orbit type partition. By Lemma \ref{zkpb6ssaaq}, $\bar{F}$ restricts to hyperk\"ahler maps on the strata, and since they are also homeomorphisms, they are hyperk\"ahler isomorphisms.
\end{proof}

We may now use the notation $\D(\g)$ instead of $\D(G)$.

\subsection{Stratification poset}\label{f80rc1x9ex}

Recall that the relation $S\leq T$ if $S\s\overline{T}$ on a partition $\cP$ of a topological space is a partial order whenever the elements of $\cP$ are locally closed. In particular, this holds for the orbit type partition of $\D(\g)$. To prove that this partition also satisfies the frontier condition (and hence is a stratification), we will first prove Theorem \ref{zt3pxhz7ur} identifying the poset of strata with the poset of root subsystems. The frontier condition will then be an easy consequence of this identification.

By Proposition \ref{0ae71l8pfy}, the $(T_K\times T_K)$-orbit type decomposition of $\D(\g)$ is the same as the $T$-orbit type decomposition of $(T^*\cO\cap(\g\times\ft^\circ))\ps/T$. It will be convenient to identify $\g^*$ with $\g$ using the non-degenerate invariant bilinear form, and hence $\ft^\circ$ with $\ft^\perp\s\g$. Then, $\D(\g)=(T\cO\cap(\g\times\ft^\perp))\ps/T$, where $T$ acts by the adjoint action on both factors. Thus, the orbit type partition of $\D(\g)$ is of the form $\cP=\{\D(\g)_{Z}:Z\in\mathcal{I}\}$ for some collection $\mathcal{I}$ of closed subgroups of $T$ (where $\D(\g)_{Z}=(T\cO\cap(\g\times\ft^\perp)\ps_{Z}/T$). Our first objective is to identify the collection $\mathcal{I}$ precisely. Let $\Phi$ be the root system of $\g$ with respect to $\ft$. We use the same notation as in \eqref{6hygdax5ph} for $Z_\Psi$ and $\D(\g)_\Psi$.

\begin{lemma}\label{6n1bilpsfe}
The stabilizer subgroup of $(X,Y)\in\g\times\g$ under the adjoint $T$-action is $Z_{\Phi(X,Y)}$, where
\[
\Phi(X,Y)=\Phi\cap\Span_\Z\{\alpha\in\Phi:(X_\alpha,Y_\alpha)\neq(0,0)\}.
\]
Moreover, $\Phi(X,Y)$ is a root subsystem of $\Phi$.
\end{lemma}

\begin{proof}
Let $\Phi'(X,Y)=\{\alpha\in\Phi:(X_\alpha,Y_\alpha)\neq(0,0)\}$. The weights spaces for the action of $T$ on $\g\times\g$ are $\g_\alpha\times\g_\alpha$ for $\alpha\in\Phi\cup\{0\}$. Thus, $t$ fixes $(X,Y)$ if and only if $t\in Z_{\Phi'(X,Y)}$, so we want to show that $Z_{\Phi'(X,Y)}=Z_{\Phi(X,Y)}$. Since $\Phi'(X,Y)\s\Phi(X,Y)$ we have $Z_{\Phi(X,Y)}\s Z_{\Phi'(X,Y)}$. Conversely, let $t\in Z_{\Phi'(X,Y)}$ and $\beta\in\Phi(X,Y)$. Then, $\beta=\sum_{\alpha\in\Phi'(X,Y)}n_\alpha\alpha$ for some $n_\alpha\in\Z$, so $\beta(t)=\prod_{\alpha\in\Phi'(X,Y)}\alpha(t)^{n_\alpha}=1$.

Any set of the form $\Phi\cap\Span_\Z \Psi$ for some subset $\Psi\s\Phi$ is a root subsystem, so $\Phi(X,Y)$ is a root subsystem.
\end{proof}

This means that there is some collection $\mathcal{J}$ of root subsystems of $\Phi$ such that the orbit type partition is $\cP=\{\D(\g)_{\Psi}:\Psi\in\mathcal{J}\}$. The next lemma shows that $\mathcal{J}$ is, in fact, the set of {\it all} root subsystems of $\Phi$.

\begin{lemma}\label{tpk0qmasjr}
$\D(\g)_{\Psi}\neq\emptyset$ for any root subsystem $\Psi\le\Phi$.
\end{lemma}

\begin{proof}
Take $0\neq X_\alpha\in\g_\alpha$ for all $\alpha\in\Psi$ and let $X=\sum_{\alpha\in\Psi}X_\alpha$. Then, $(\tau,X)\in\allowbreak(T\cO\cap(\g\times\ft^\perp))$. Recall that a point is polystable if and only if $0$ is in the interior of the convex hull of its set of weights \cite[Proposition 6.15]{pop94}. Thus, $(\tau,X)$ is polystable since its set of weights is precisely $\Psi\cup\{0\}$ and we have $\alpha\in\Psi\Longleftrightarrow-\alpha\in\Psi$. Moreover, $\Phi(\tau,X)=\Phi\cap\Span_\Z\Psi=\Psi$, so $(\tau,X)\in(T\cO\cap(\g\times\ft^\perp))\ps_{Z_\Psi}/T=\D(\g)_{\Psi}$.
\end{proof}

In other words, we have shown that the map $\Psi\mto\D(\g)_\Psi$ in Theorem \ref{zt3pxhz7ur} is surjective. The next lemma shows that it is injective.

\begin{lemma}\label{x8jufuo3f0}
Let $\Psi_1$ and $\Psi_2$ be two root subsystems. Then, $\D(\g)_{\Psi_1}=\D(\g)_{\Psi_2}$ if and only if $\Psi_1=\Psi_2$.
\end{lemma}

\begin{proof}
The only non-trivial part is to show that if $Z_{\Psi_1}=Z_{\Psi_2}$ then $\Psi_1=\Psi_2$. It suffices to show that if $\Psi$ is a root subsystem and $\alpha\in\Phi$ is such that $\alpha(t)=1$ for all $t\in Z_\Psi$, then $\alpha\in\Psi$. By viewing $\alpha$ as an element of $\ft^*$, we want to show that if $\alpha(H)\in 2\pi i \Z$ for all $H\in \ft$ such that $\{\beta(H):\beta\in\Psi\}\s 2\pi i\Z$, then $\alpha\in\Psi$. Let $\beta_1,\ldots,\beta_k\in\Psi$ be a set of simple roots and complete it to a basis $\{\beta_1,\ldots,\beta_n\}$ for $\ft^*$. Let $H_1,\ldots,H_n$ be the dual basis in $\ft$. Then, $\alpha=a_1\beta_1+\cdots+a_n\beta_n$ and since $\beta_i(2\pi iH_j)\in 2\pi i\Z$, the assumption on $\alpha$ implies that $2\pi ia_j=\alpha(2\pi iH_j)\in 2\pi i\Z$ so $a_j\in\Z$. Moreover, for all $H\in\Span\{H_{k+1},\ldots,H_n\}$ we have $\beta(H)=0$ for all $\beta\in\Psi$, so $\alpha(H)=(a_{k+1}\beta_{k+1}+\cdots+a_n\beta_n)(H)\in 2\pi i\Z$. Thus, $a_{k+1}\beta_{k+1}+\cdots+a_n\beta_n=0$ and hence $\alpha\in(\Span_\Z\Psi)\cap\Phi=\Psi$.
\end{proof}

Thus, $\Psi\mto\D(\g)_\Psi$ is a bijection from the set of root subsystems of $\Phi$ to the set of orbit type strata of $\D(\g)$. It remains to show that it is an isomorphism of posets. The crucial ingredient is the following result:

\begin{proposition}\label{b65xnhncei}
For all $\Psi\le\Phi$ we have
\[
\overline{\D(\g)_{\Psi}}=\bigcup_{\chi\le\Psi}\D(\g)_{\chi}.
\]
Moreover, $\overline{\D(\g)_\Psi}$ is also the Zariski-closure of $\D(\g)_\Psi$, and $\D(\g)_\Psi$ is Zariski-locally-closed.
\end{proposition}

\begin{proof}
Let $\g'_\Psi=\ft\oplus\bigoplus_{\alpha\in\Psi}\g_\alpha$ (which is slightly different than the $\g_\Psi$ defined in \S\ref{nj5e2b9c36}), $M=T\cO\cap(\g\times\ft^\perp)$ and $M_\Psi=M\cap(\g'_\Psi\times\g'_\Psi)$. We claim that $M_{Z_\Psi}\s M_\Psi$. Indeed, if $(X,Y)\in M_{Z_\Psi}$ then by Lemma \ref{6n1bilpsfe}, $Z_\Psi=Z_{\Phi(X,Y)}$ and by the proof of Lemma \ref{x8jufuo3f0} this implies $\Psi=\Phi(X,Y)$. Thus, $\{\alpha\in\Phi:(X_\alpha,Y_\alpha)\ne(0,0)\}\s\Psi$, so $(X,Y)\in M_\Psi$. Hence, $\D(\g)_\Psi=M\ps_{Z_\Psi}/T=(M_\Psi)\ps_{Z_\Psi}/T$. Moreover, $(M_\Psi)\ps_{Z_\Psi}$ is Zariski-open in $M_\Psi$ since $(M_\Psi)\ps_{Z_\Psi}=(M_\Psi)\st\cap (M_\Psi)_{Z_\Psi}$, where $(M_\Psi)\st$ is the set of stable points (which is always Zariski-open) and $Z_\Psi$ is the kernel of the action of $T$ on $M_\Psi$ so $(M_\Psi)_{Z_\Psi}$ is Zariski-open in $M_\Psi$ (see \cite[Proposition 7.2]{pop94}). Hence, $\overline{\D(\g)_\Psi}=M_\Psi\ps/T=\Spec(\C[M_\Psi]^T)$ is Zariski-closed in $\D(\g)$, and so $\overline{\D(\g)_\Psi}$ is the Zariski-closure of $\D(\g)_\Psi$. Since $(M_\Psi)\ps_{Z_\Psi}$ is Zariski-open in $M_\Psi$, so is its image in $M\ps_\Psi/T$, so $\D(\g)_\Psi$ is Zariski-locally-closed. Now, by decomposing $M_\Psi\ps/T$ by $T$-orbit types, we get
\[\overline{\D(\g)_\Psi}=\bigcup_{\chi\le \Psi} (M_\Psi)\ps_{Z_\chi}/T=\bigcup_{\chi\le\Psi}M\ps_{Z_\chi}/T=\bigcup_{\chi\le\Psi}\D(\g)_\chi.\qedhere\]
\end{proof}

In particular, this proposition shows that $\Psi_1\le\Psi_2$ if and only if $\D(\g)_{\Psi_1}\le\D(\g)_{\Psi_2}$, so the map $\Psi\mto\D(\g)_\Psi$ is an isomorphism of posets, and we have proved Theorem \ref{zt3pxhz7ur}. Moreover, it now follows easily that the partition is a stratification:

\begin{proof}[Proof of Theorem \ref{2rzpnjamfd}]
In general, the orbit type decomposition of a hyperk\"ahler quotient satisfies the manifold and local conditions, so it suffices to show frontier condition. Suppose that $\D(\g)_{\Psi_1}\cap\overline{\D(\g)_{\Psi_2}}\ne\emptyset$ for some $\Psi_1,\Psi_2\le\Phi$. By Proposition \ref{b65xnhncei}, this implies $\D(\g)_{\Psi_1}\cap\D(\g)_\chi\ne\emptyset$ for some $\chi\le\Psi_2$, and hence $\D(\g)_{\Psi_1}=\D(\g)_\chi$. By Lemma \ref{x8jufuo3f0}, we get $\Psi_1=\chi$ so $\D(\g)_{\Psi_1}\s\overline{\D(\g)_{\Psi_2}}$. 
\end{proof}

\subsection{Description of the strata}\label{g3db9khyew}

Now that we know that the orbit type partition of $\D(\g)$ is of the form $\{\D(\g)_\Psi:\Psi\le\Phi\}$, our next goal is to describe the strata $\D(\g)_\Psi$ individually. 

First note that since $\Phi$ is a greatest element in the poset of root subsystems, $\D(\g)_\Phi$ is dense in $\D(\g)$. We denote this top stratum by $\D(\g)_{\mathrm{top}}:=\D(\g)_\Phi$. 

\begin{proposition}\label{n25r7847fc}
The stratum $\D(\g)_{\mathrm{top}}$ is Zariski-open, connected, and of real dimension $4(\dim\g-2\rk\g)$.
\end{proposition}

\begin{proof}
By Proposition \ref{b65xnhncei}, $\D(\g)_{\mathrm{top}}$ is Zariski-locally-closed in $\D(\g)$. But its Zariski-closure is the same as its standard closure which is $\D(\g)$, so $\D(\g)_{\mathrm{top}}$ is Zariski-open in $\D(\g)$. Since $\D(\g)$ is a GIT quotient of the connected space $T\cO\cap(\g\times\ft^\perp)$, it is connected. Thus, $\D(\g)_{\mathrm{top}}$ is connected since it is Zariski-open in $\D(\g)$. The set of stable points of $T\cO\cap(\g\times\ft^\perp)$ is non-empty (for example, it contains $(\tau,X)$ if $X_\alpha\neq 0$ for all $\alpha\in\Phi$), so the complex algebraic dimension of $\D(\g)$ is $\dim_\C(T\cO\cap(\g\times\ft^\perp))-\rk\g$. Moreover, the smooth map $T\cO\to\ft$, $(X,Y)\mto \pi_\ft(Y)$ (where $\pi_\ft$ is the orthogonal projection $\g\to\ft$) has full rank at a generic point, so
\[
\dim_\C(T\cO\cap(\g\times\ft^\perp))=\dim_\C(T\cO)-\rk\g=2\dim\g-3\rk\g,
\]
and hence $\dim_\C\D(\g)=2\dim\g-4\rk\g$. Since $\D(\g)_{\mathrm{top}}$ is open in $\D(\g)$, it has the same complex dimension.
\end{proof}

Let $\Psi$ be a root subsystem of $\Phi$, $W_\Psi$ be the Weyl group of $\Psi$, and $|W_\Phi:W_\Psi|$ the index of $W_\Psi$ in $W_\Phi$. The rest of this section is devoted to the proof that $\D(\g)_\Psi$ is a disjoint union of $|W_\Phi:W_\Psi|$ copies of $\D(\g_\Psi)_{\mathrm{top}}$ (Theorem \ref{6bfrjygjvw}).

\begin{lemma}
The connected subgroup $G_\Psi$ of $G$ with Lie algebra $\g_\Psi$ is closed and has a compact real form $K_\Psi\s K$.
\end{lemma}

\begin{proof}
Since the Killing form of $\g$ remains nondegenerate on $\g_\Psi$ (Proposition \ref{utfyudsumb}) we have $\g=\g_\Psi\oplus\g_\Psi^\perp$ and it follows that $\g_\Psi=(\fk_\Psi)_\C$ where $\fk_\Psi=\g_\Psi\cap\fk$. Moreover, semisimple subgroups of compact Lie groups are closed, so the connected subgroup $K_\Psi$ of $K$ with Lie algebra $\fk_\Psi$ is closed. Hence, $G_\Psi:=(K_\Psi)_\C$ is a closed connected subgroup of $G$ with Lie algebra $\g_\Psi$. 
\end{proof}

Recall that $\ft_\Psi$ is the span of the coroots of $\Psi$ and is a Cartan subalgebra of $\g_\Psi$. Let $\mathfrak{z}_\Psi=\Lie(Z_\Psi)=\{h\in\ft:\alpha(h)=0,\forall \alpha\in\Psi\}$ so that $\ft=\mathfrak{z}_\Psi\oplus\ft_\Psi$. Then, $\mathfrak{z}_\Psi\oplus\g_\Psi=\allowbreak\ft\oplus\bigoplus_{\alpha\in\Psi}\g_\alpha$ and is a reductive Lie algebra with semisimple factor $\g_\Psi$. 

\begin{lemma}\label{4b4w3eltv5}
The intersection $\cO\cap(\mathfrak{z}_\Psi\oplus\g_\Psi)$ has $|W_\Phi:W_\Psi|$ connected components, each of the form $\zeta+\cO_\Psi$ for some $\zeta\in\mathfrak{z}_\Psi$ and $\cO_\Psi$ a regular semisimple $G_\Psi$-orbit in $\g_\Psi$.
\end{lemma}

\begin{proof}
Since $\cO\cap(\mathfrak{z}_\Psi\oplus\g_\Psi)$ is $G_\Psi$-invariant, it is a union of $G_\Psi$-orbits. But $G_\Psi$ acts trivially on $\mathfrak{z}_\Psi$, so they are of the form $\zeta+\cO_\Psi$ where $\zeta\in\mathfrak{z}_\Psi$ and $\cO_\Psi$ is a regular semisimple $G_\Psi$-orbit in $\g_\Psi$. From the classification of semisimple orbits \cite[\S2.2]{col93}, $\cO_\Psi$ intersects $\ft_\Psi$ in precisely $|W_\Psi|$ elements. Hence $(\zeta+\cO_\Psi)\cap\ft$ contains exactly $|W_\Psi|$ elements. But $(\cO\cap(\mathfrak{z}_\Psi\oplus\g_\Psi))\cap\ft=\cO\cap\ft$ contains exactly $|W_\Phi|$ elements, so there must be exactly $|W_\Phi|/|W_\Psi|=|W_\Phi:W_\Psi|$ of these $G_\Psi$-orbits. Moreover, semisimple orbits are closed, so $\zeta+\cO_\Psi$ is closed in $\cO\cap(\mathfrak{z}_\Psi\oplus\g_\Psi)$. Since there are only finitely many of them, they are also open in $\cO\cap(\mathfrak{z}_\Psi\oplus\g_\Psi)$, and hence they are the connected components.
\end{proof}

\begin{proposition}\label{ih2wo35ypo}
For all $\Psi\le\Phi$, $\D(\g)_\Psi$ is isomorphic as a hyperk\"ahler manifold to a disjoint union of $|W_\Phi:W_\Psi|$ copies of $\D(\g_\Psi)_{\mathrm{top}}$. 
\end{proposition}

\begin{proof}
Let $T_{K_\Psi}$ be the maximal torus in $K_\Psi$ with Lie algebra $\ft_\Psi\cap\fk$ and $T_\Psi=(T_{K_\Psi})_\C$ the maximal torus in $G_\Psi$ with Lie algebra $\ft_\Psi$. Recall from the proof of Proposition \ref{b65xnhncei} that $\D(\g)_\Psi=(M_\Psi)\ps_{Z_\Psi}/T$, where $M=T\cO\cap(\g\times\ft^\perp)$ and $M_\Psi=M\cap((\mathfrak{z}_\Psi\oplus\g_\Psi)\times(\mathfrak{z}_\Psi\oplus\g_\Psi))$. Thus, by Lemma \ref{4b4w3eltv5}, $\D(\g)_\Psi$ is a disjoint union of sets of the form $(T(\zeta+\cO_\Psi)\cap(\g\times\ft^\perp))\ps_{Z_\Psi}/T$ for $\zeta\in\mathfrak{z}_\Psi$ and $\cO_\Psi\s\g_\Psi$. These sets are homeomorphic to the top stratum $\D(\g_\Psi)_{\mathrm{top}}=\allowbreak(T\cO_\Psi\cap(\g_\Psi\times\ft_\Psi^\perp))\ps_{Z_\Psi\cap T_\Psi}/T_\Psi$ via the map $(X,Y)\mto(\zeta+X,Y)$, and hence it suffices to show this map is hyperk\"ahler. We will show that it descends from a hyperk\"ahler map $T^*G_\Psi\to T^*G$ and conclude that it is hyperk\"ahler by Lemma \ref{zkpb6ssaaq}.

Since $\cO_\Psi$ is semisimple, we have $\cO_\Psi=G_\Psi\cdot h$ for some $h\in \ft_\Psi$. Thus, $\zeta+h\in\allowbreak(\zeta+\cO_\Psi)\cap\ft\s\cO\cap\ft$ and hence $\zeta+h=w\cdot\tau$ for some $w\in W$. Let $k\in N_K(T_K)$ be a representative of $w$. Then, the composition
\[
F:G_\Psi\times\g_\Psi\hooklongrightarrow G\times \g\too G\times\g,\quad(g,X)\mtoo(gk,\Ad_{k^{-1}}X)
\]
is hyperk\"ahler since this is just the action of $(1,k^{-1})\in K\times K$. Moreover, it is equivariant with respect to the Lie group homomorphism $T_{K_\Psi}\times T_{K_\Psi}\to T_K\times T_K$, $(s,t)\mto(s,k^{-1}tk)$ and maps the zero level set of the moment for $T_{K_\Psi}\times T_{K_\Psi}$ to that of $T_K\times T_K$. Therefore, Lemma \ref{zkpb6ssaaq} applies and we get a continuous map $\bar{F}:\D(\g_\Psi)\to\D(\g)$ which is hyperk\"ahler on the strata that it preserves. For all $g\in G_\Psi$ we have $\zeta+\Ad_gh=\Ad_g(\zeta+h)=\Ad_{gk}\tau$, so the diagram
\[
\begin{tikzcd}
T^*G_\Psi\slll{}(T_{K_\Psi}\times T_{K_\Psi}) \arrow{r}{\bar{F}}\arrow{d}{\cong} & T^*G\slll{}(T_K\times T_K) \arrow{d}{\cong}\\
(T\cO_\Psi\cap(\g_\Psi\times\ft_\Psi^\perp))\ps/T_\Psi \arrow{r} & (T\cO\cap(\g\times\ft^\perp))\ps/T
\end{tikzcd}
\]
commutes, where the bottom map is $(\Ad_gh,\Ad_gX) \mto (\zeta+\Ad_gh,\Ad_gX)$ and the vertical maps are the usual homeomorphism as in Proposition \ref{0ae71l8pfy}. The restriction of the bottom map to $\D(\g_\Psi)_{\mathrm{top}}$ is the map considered earlier, and hence $\D(\g_\Psi)_{\mathrm{top}}$ is hyperk\"ahler isomorphic to the connected components of $\D(\g)_\Psi$ corresponding to $\zeta+\cO_\Psi$.
\end{proof}

\begin{corollary}\label{atk1mpq9ez}
The stratum $\D(\g)_{\mathrm{bottom}}:=\D(\g)_{\emptyset}$ is a finite set of $|W_\Phi|$ elements.
\end{corollary}

\begin{proof}
By Proposition \ref{ih2wo35ypo}, $\D(\g)_{\mathrm{bottom}}$ is a disjoint union of $|W_\Phi|$ copies of $\D(\g_\emptyset)_{\mathrm{top}}\allowbreak=\D(0)_{\mathrm{top}}$, and $\D(0)_{\mathrm{top}}$ is just a point by Proposition \ref{n25r7847fc}.
\end{proof}

Proposition \ref{n25r7847fc}, Proposition \ref{ih2wo35ypo} and Corollary \ref{atk1mpq9ez} together prove Theorem \ref{6bfrjygjvw}.

\subsection{A coarser stratification}

The goal of this section is to prove Theorem \ref{xvislcts8s} about the coarser partition $\cP=\{\D(\g)_{[\h]}:[\h]\in\mathcal{C}_\g\}$. Let $[\h]=[\g_\Psi]\in\mathcal{C}_\g$ where $\Psi\le \Phi$. Then, by Proposition \ref{cv6qokpmij}, we have
\[
\D(\g)_{[\g_\Psi]}=\bigcup_{w\in W_\Phi}\D(\g)_{w\cdot\Psi}.
\]
Let $\{w\cdot\Psi:w\in W_\Phi\}=\{w_1\cdot \Psi,\ldots,w_n\cdot\Psi\}$, where the $w_i\cdot \Psi$'s are distinct. Then, $n|W_\Phi:W_\Psi|$ is the embedding number $m_\g(\g_\Psi)$ introduced in \S\ref{nj5e2b9c36}, and
\[
\D(\g)_{[\g_\Psi]}=\bigcup_{i=1}^n\D(\g)_{w_i\cdot\Psi}.
\]

\begin{lemma}\label{0vk6s011tr}
The above union is a topological disjoint union.
\end{lemma}

\begin{proof}
It suffices to show that if $u,v\in W_\Phi$ and $\overline{\D(\g)_{u\cdot\Psi}}\cap\D(\g)_{v\cdot\Psi}\neq\emptyset$ then $u\cdot\Psi=v\cdot\Psi$. By Proposition \ref{b65xnhncei}, we have
\[
\overline{\D(\g)_{u\cdot\Psi}}\cap\D(\g)_{v\cdot\Psi}= \bigcup_{\chi\le u\cdot\Psi}\D(\g)_{\chi}\cap\D(\g)_{v\cdot\Psi},
\]
so there exists $\chi\le u\cdot \Psi$ such that $\D(\g)_\chi\cap \D(\g)_{v\cdot\Psi}\neq\emptyset$. Hence, $v\cdot\Psi=\chi\s u\cdot \Psi$ so $v\cdot \Psi=u\cdot\Psi$. 
\end{proof}

In particular, Lemma \ref{0vk6s011tr} says that each piece $\D(\g)_{[\h]}$ in $\cP$ is a topological manifold and is locally closed. Moreover, by combining with Proposition \ref{ih2wo35ypo}, we get that $\D(\g)_{[\h]}$ has $m_\g(\h)$ connected components, each isomorphic to $\D(\h)_{\mathrm{top}}$ as a hyperk\"ahler manifold. Now, recall that $\mathcal{C}_\g$ has a partial order $\le$ induced by inclusion.

\begin{lemma}\label{qq0ms1ekr6}
For all $[\h]\in\mathcal{C}_\g$, we have
\[
\overline{\D(\g)_{[\h]}}=\bigcup_{[\q]\leq[\h]}\D(\g)_{[\q]}.
\]
\end{lemma}

\begin{proof}
Let $\Psi$ be such that $[\h]=[\g_\Psi]$. If we let $[\Psi]$ be the equivalence class of $\Psi$ in $\{\text{root subsystems of $\Phi$}\}/W_\Phi$, we get
\begin{align*}
\overline{\D(\g)_{[\g_\Psi]}} &= \bigcup_{w\in W} \overline{\D(\g)_{w\cdot \Psi}}=\bigcup_{w\in W} \bigcup_{\chi\le w\cdot\Psi}\D(\g)_\chi=\bigcup_{[\chi]\le[\Psi]}\bigcup_{w\in W}\D(\g)_{w\cdot\chi} \\
&= \bigcup_{[\chi]\le[\Psi]}\D(\g)_{[\g_\chi]} = \bigcup_{[\q]\leq[\h]}\D(\g)_{[\q]}.\qedhere
\end{align*}
\end{proof}

It is then immediate that $[\h_1]\le[\h_2]$ if and only if $\D(\g)_{[\h_1]}\le\D(\g)_{[\h_2]}$, so the map $\mathcal{C}_\g\to\cP$, $[\h]\mto\D(\g)_{[\h]}$ is an isomorphism of posets. Moreover, the frontier condition now follows easily as in \S\ref{f80rc1x9ex}. This concludes the proof of Theorem \ref{xvislcts8s}.

\section{Examples}\label{aejay646nw}

In this section we look at specific examples of $\D(\g)$ and describe their stratification into hyperk\"ahler manifolds explicitly. We will focus on the coarser stratification whose set of strata is isomorphic to the partially ordered set $\mathcal{C}_\g$ of conjugacy classes of regular semisimple subalgebras (Theorem \ref{xvislcts8s}). We will draw Hasse diagrams where each node is of the form $mL$ where $L$ is the Lie algebra isomorphism class of a conjugacy class $[\h]\in\mathcal{C}_\g$ and $m$ its embedding number. In other words, a node of the form $mL$ in the Hasse diagram represents a stratum which is isomorphic as a hyperk\"ahler manifold to a disjoint union of $m$ copies of $\D(L)_{\mathrm{top}}$. We write isomorphism classes of Lie algebras multiplicatively, e.g.\ $L=A_1^2B_2$ is $\mathfrak{sl}(2,\C)\oplus\mathfrak{sl}(2,\C)\oplus\mathfrak{so}(5,\C)$.

\subsection{$\fsl(2,\C)$}

The root system $\Phi$ of $\fsl(2,\C)$ embeds in $\R^2$ as:
\begin{center}
\begin{tikzpicture}[xscale=0.5, yscale=0.5]
    \foreach\ang in {0,180}{
     \draw (0,0) -- (\ang:2cm);
     \node[draw, circle, fill, inner sep=1pt] at (\ang:2cm) {};
    }
    \fill (0,0) circle [radius=1pt];
\end{tikzpicture}
\end{center}
Thus, the root subsystems are just $\Phi$ itself and the empty set. They both determine a conjugacy class in $\mathcal{C}_\g$, and the embedding numbers are $m(\fsl(2,\C))=1$ and $m(0)=|W|=2$, so the Hasse diagram is:
\[
\begin{tikzpicture}\ssmall
\node (0) at (0, 1) {$A_1$};
\node (1) at (0, 0) {$2$};
\draw (1) -- (0);
\end{tikzpicture}
\]
We can see from this diagram that $\D(\fsl(2,\C))$ consists of a smooth open dense subset and two isolated singularities. This agrees with Dancer's computation \cite[pp.\ 88--99]{dan93} that $\D(\fsl(2,\C))$ is the $D_2$-surface, i.e.\ the affine variety in $\C^3$ cut out by the equation $x^2-zy^2=z$.

\subsection{The exceptional Lie algebra $\g_2$}

The root system of $\g_2$ embeds in $\R^2$ as:
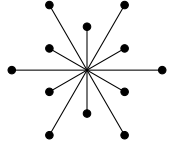
\begin{center}
\begin{tikzpicture}[xscale=0.5, yscale=0.5]
    \foreach\ang in {60,120,...,360}{
     \draw (0,0) -- (\ang:2cm);
     \node[draw, circle, fill, inner sep=1pt] at (\ang:2cm) {};
    }
    \foreach\ang in {30,90,...,330}{
     \draw (0,0) -- (\ang:1.155cm);
     \node[draw, circle, fill, inner sep=1pt] at (\ang:1.155cm) {};
    }
\end{tikzpicture}
\end{center}
The six long roots form an $A_2$ subsystem:
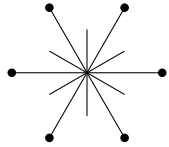
\begin{center}
\begin{tikzpicture}[xscale=0.5, yscale=0.5]
    \foreach\ang in {60,120,...,360}{
     \draw (0,0) -- (\ang:2cm);
     \node[draw, circle, fill, inner sep=1pt] at (\ang:2cm) {};
    }
    \foreach\ang in {30,90,...,330}{
     \draw (0,0) -- (\ang:1.155cm);
    }
\end{tikzpicture}
\end{center}
It is fixed by the Weyl group, so the embedding number is $|W_{G_2}|/|W_{A_2}|=12/6=2$. There are three $A_1^2$ subsystems:
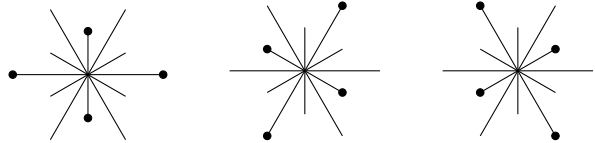
\begin{center}
\begin{tikzpicture}[xscale=0.5, yscale=0.5]
    \foreach\ang in {60,120,...,360}{
     \draw (0,0) -- (\ang:2cm);
    }
    \foreach\ang in {30,90,...,330}{
     \draw (0,0) -- (\ang:1.155cm);
    }
    \node[draw, circle, fill, inner sep=1pt] at (0:2cm) {};
    \node[draw, circle, fill, inner sep=1pt] at (90:1.155cm) {};
    \node[draw, circle, fill, inner sep=1pt] at (180:2cm) {};
    \node[draw, circle, fill, inner sep=1pt] at (270:1.155cm) {};
\end{tikzpicture}
\qquad
\begin{tikzpicture}[xscale=0.5, yscale=0.5]
    \foreach\ang in {60,120,...,360}{
     \draw (0,0) -- (\ang:2cm);
    }
    \foreach\ang in {30,90,...,330}{
     \draw (0,0) -- (\ang:1.155cm);
    }
    \node[draw, circle, fill, inner sep=1pt] at (60:2cm) {};
    \node[draw, circle, fill, inner sep=1pt] at (150:1.155cm) {};
    \node[draw, circle, fill, inner sep=1pt] at (240:2cm) {};
    \node[draw, circle, fill, inner sep=1pt] at (330:1.155cm) {};
\end{tikzpicture}
\qquad
\begin{tikzpicture}[xscale=0.5, yscale=0.5]
    \foreach\ang in {60,120,...,360}{
     \draw (0,0) -- (\ang:2cm);
    }
    \foreach\ang in {30,90,...,330}{
     \draw (0,0) -- (\ang:1.155cm);
    }
    \node[draw, circle, fill, inner sep=1pt] at (120:2cm) {};
    \node[draw, circle, fill, inner sep=1pt] at (30:1.155cm) {};
    \node[draw, circle, fill, inner sep=1pt] at (300:2cm) {};
    \node[draw, circle, fill, inner sep=1pt] at (210:1.155cm) {};
\end{tikzpicture}
\end{center}
They are in the same $W$-orbit, so they determine a single conjugacy class in $\mathcal{C}_{\g_2}$ with embedding number $3|W_{G_2}|/|W_{A_1^2}|=3\cdot 12/4=9$. There are three $A_1$ subsystems formed by the long roots:
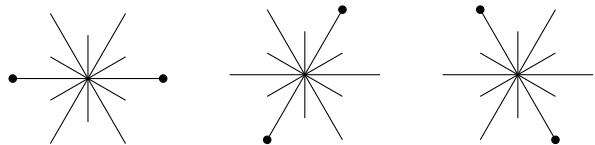
\begin{center}
\begin{tikzpicture}[xscale=0.5, yscale=0.5]
    \foreach\ang in {60,120,...,360}{
     \draw (0,0) -- (\ang:2cm);
    }
    \foreach\ang in {30,90,...,330}{
     \draw (0,0) -- (\ang:1.155cm);
    }
    \node[draw, circle, fill, inner sep=1pt] at (0:2cm) {};
    \node[draw, circle, fill, inner sep=1pt] at (180:2cm) {};
\end{tikzpicture}
\qquad
\begin{tikzpicture}[xscale=0.5, yscale=0.5]
    \foreach\ang in {60,120,...,360}{
     \draw (0,0) -- (\ang:2cm);
    }
    \foreach\ang in {30,90,...,330}{
     \draw (0,0) -- (\ang:1.155cm);
    }
    \node[draw, circle, fill, inner sep=1pt] at (60:2cm) {};
    \node[draw, circle, fill, inner sep=1pt] at (240:2cm) {};
\end{tikzpicture}
\qquad
\begin{tikzpicture}[xscale=0.5, yscale=0.5]
    \foreach\ang in {60,120,...,360}{
     \draw (0,0) -- (\ang:2cm);
    }
    \foreach\ang in {30,90,...,330}{
     \draw (0,0) -- (\ang:1.155cm);
    }
    \node[draw, circle, fill, inner sep=1pt] at (120:2cm) {};
    \node[draw, circle, fill, inner sep=1pt] at (300:2cm) {};
\end{tikzpicture}
\end{center}
They are in the same $W$-orbit and the embedding number is $3\cdot 12 / 2 = 18$. Similarly, the short roots form three $A_1$ subsystems
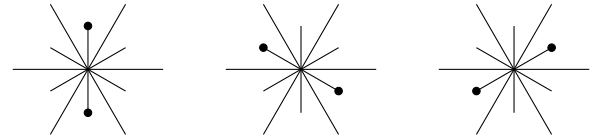
\begin{center}
\begin{tikzpicture}[xscale=0.5, yscale=0.5]
    \foreach\ang in {60,120,...,360}{
     \draw (0,0) -- (\ang:2cm);
    }
    \foreach\ang in {30,90,...,330}{
     \draw (0,0) -- (\ang:1.155cm);
    }
    \node[draw, circle, fill, inner sep=1pt] at (90:1.155cm) {};
    \node[draw, circle, fill, inner sep=1pt] at (270:1.155cm) {};
\end{tikzpicture}
\qquad
\begin{tikzpicture}[xscale=0.5, yscale=0.5]
    \foreach\ang in {60,120,...,360}{
     \draw (0,0) -- (\ang:2cm);
    }
    \foreach\ang in {30,90,...,330}{
     \draw (0,0) -- (\ang:1.155cm);
    }
    \node[draw, circle, fill, inner sep=1pt] at (150:1.155cm) {};
    \node[draw, circle, fill, inner sep=1pt] at (330:1.155cm) {};
\end{tikzpicture}
\qquad
\begin{tikzpicture}[xscale=0.5, yscale=0.5]
    \foreach\ang in {60,120,...,360}{
     \draw (0,0) -- (\ang:2cm);
    }
    \foreach\ang in {30,90,...,330}{
     \draw (0,0) -- (\ang:1.155cm);
    }
    \node[draw, circle, fill, inner sep=1pt] at (30:1.155cm) {};
    \node[draw, circle, fill, inner sep=1pt] at (210:1.155cm) {};
\end{tikzpicture}
\end{center}
which are $W$-conjugate and have embedding number $18$. Finally, the empty set
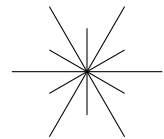
\begin{center}
\begin{tikzpicture}[xscale=0.5, yscale=0.5]
    \foreach\ang in {60,120,...,360}{
     \draw (0,0) -- (\ang:2cm);
    }
    \foreach\ang in {30,90,...,330}{
     \draw (0,0) -- (\ang:1.155cm);
    }
\end{tikzpicture}
\end{center}
forms a root subsystem with embedding number $|W_{G_2}|=12$. Therefore, the stratification structure of $\D(\g_2)$ can be written symbolically as:
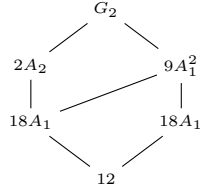
\[
\begin{tikzpicture}[baseline=(current bounding box.center), yscale=.75]\ssmall
\node (0) at (1, 3) {$G_2$};
\node (3) at (0, 2) {$2A_2$};
\node (1) at (2, 2) {$9A_1^2$};
\node (5) at (0, 1) {$18A_1$};
\node (4) at (2, 1) {$18A_1$};
\node (2) at (1, 0) {$12$};
\draw (3) -- (0);
\draw (2) -- (4);
\draw (5) -- (1);
\draw (2) -- (5);
\draw (4) -- (1);
\draw (1) -- (0);
\draw (5) -- (3);
\end{tikzpicture}
\]
For example, we see from this diagram that on the boundary of the open dense stratum $\D(\g_2)_{\mathrm{top}}$ there are two copies of the space $\D(\fsl(3,\C))_{\mathrm{top}}$. Also, the set of most singular points is a finite set of $12$ elements.

\subsection{All simple Lie algebras of rank $\le 4$}\label{ixv2fty1a1}

The computations outlined for $\fsl(2,\C)$ and $\g_2$ can be implemented on a computer. We have done that for all simple Lie algebras of rank $\le 4$. The diagrams are arranged so that a node $m'L'$ is higher than a node $mL$ if and only if $\dim\D(L')_{\mathrm{top}}>\dim \D(L)_{\mathrm{top}}$.

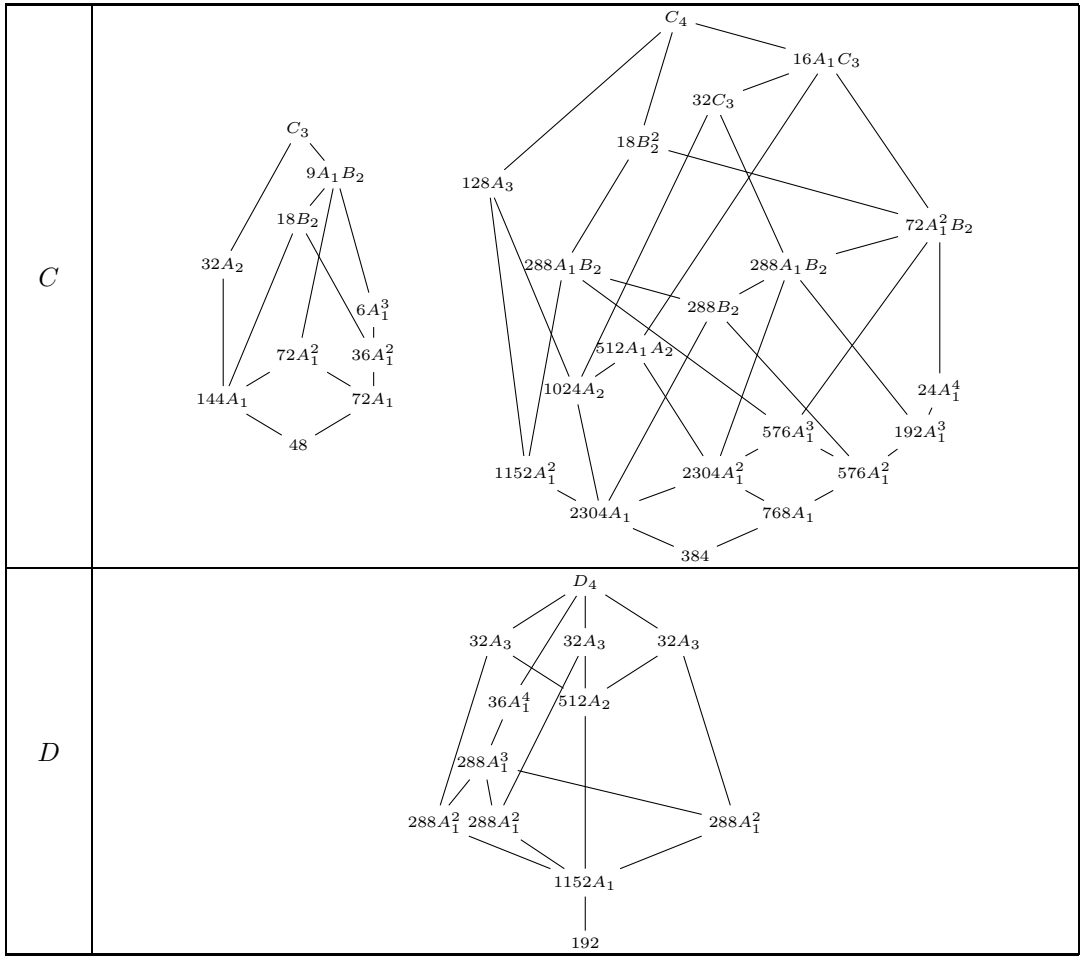
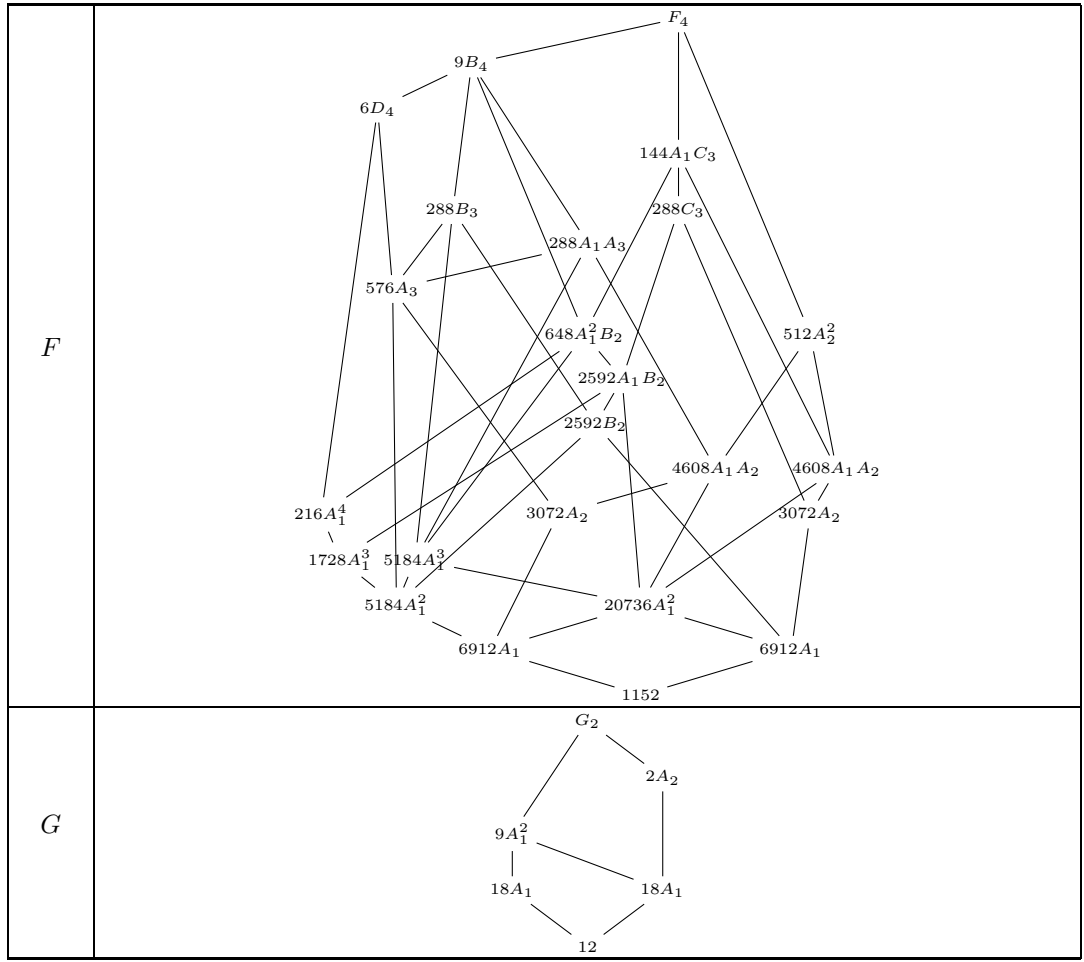
\setlength\LTleft{-0.75cm}
\begin{longtable}{|c|c|}
\hline
Type & Hasse diagrams \\
 \hline
$A$ & \begin{tikzpicture}[baseline=(current bounding box.center)]\ssmall
\node (0) at (0, 1) {$A_1$};
\node (1) at (0, 0) {$2$};
\draw (1) -- (0);
\end{tikzpicture}
 \qquad \begin{tikzpicture}[yscale=1,baseline=(current bounding box.center)]\ssmall
\node (0) at (0, 2) {$A_2$};
\node (1) at (0, 1) {$9A_1$};
\node (2) at (0, 0) {$6$};
\draw (1) -- (0);
\draw (2) -- (1);
\end{tikzpicture} \qquad 
\begin{tikzpicture}[xscale=1.0, yscale=0.9, baseline=(current bounding box.center)]\ssmall
\node (2) at (0, 4) {$A_3$};
\node (1) at (-1, 3) {$16A_2$};
\node (3) at (1, 2) {$18A_1^2$};
\node (4) at (0, 1) {$72A_1$};
\node (0) at (0, 0) {$24$};
\draw (1) -- (2);
\draw (3) -- (2);
\draw (4) -- (1);
\draw (4) -- (3);
\draw (0) -- (4);
\end{tikzpicture} \qquad 
\begin{tikzpicture}[xscale=0.8, yscale=0.65, baseline=(current bounding box.center)]\ssmall
\node (1) at (1, 6) {$A_4$};
\node (5) at (2, 5) {$25A_3$};
\node (0) at (0, 4) {$100A_1A_2$};
\node (3) at (0, 3) {$200A_2$};
\node (2) at (2, 2) {$450A_1^2$};
\node (4) at (1, 1) {$600A_1$};
\node (6) at (1, 0) {$120$};
\draw (0) -- (1);
\draw (6) -- (4);
\draw (3) -- (0);
\draw (2) -- (0);
\draw (4) -- (3);
\draw (4) -- (2);
\draw (5) -- (1);
\draw (2) -- (5);
\draw (3) -- (5);
\end{tikzpicture} \\
\hline
$B$ &  \begin{tikzpicture}[xscale=0.8,baseline=(current bounding box.center)]\ssmall
\node (2) at (1, 3) {$B_2$};
\node (1) at (2, 2) {$2A_1^2$};
\node (4) at (2, 1) {$8A_1$};
\node (3) at (0, 1) {$8A_1$};
\node (0) at (1, 0) {$8$};
\draw (1) -- (2);
\draw (3) -- (2);
\draw (4) -- (1);
\draw (0) -- (3);
\draw (0) -- (4);
\end{tikzpicture} \enskip \begin{tikzpicture}[xscale=1.2,yscale=0.7,baseline=(current bounding box.center)]\ssmall
\node (3) at 	(3, 7) {$B_3$};
\node (2) at 	(3.5, 6) {$2A_3$};
\node (0) at 	(2.5, 5) {$18B_2$};
\node (9) at 	(4.2, 4) {$32A_2$};
\node (4) at 	(3, 3) {$18A_1^3$};
\node (7) at 	(3.8, 2) {$36A_1^2$};
\node (5) at 	(3, 2) {$72A_1^2$};
\node (8) at 	(4, 1) {$144A_1$};
\node (6) at 	(2, 1) {$72A_1$};
\node (1) at 	(3, 0) {$48$};
\draw (7) -- (4);
\draw (5) -- (4);
\draw (7) -- (0);
\draw (4) -- (3);
\draw (9) -- (2);
\draw (6) -- (0);
\draw (8) -- (9);
\draw (1) -- (8);
\draw (1) -- (6);
\draw (6) -- (5);
\draw (8) -- (7);
\draw (2) -- (3);
\draw (0) -- (3);
\draw (8) -- (5);
\draw (7) -- (2);
\end{tikzpicture} \enskip \begin{tikzpicture}[xscale=1,yscale=0.55,baseline=(current bounding box.center)]\ssmall
\node (17) at (4, 13) {$B_4$};
\node (16) at (2, 12) {$2D_4$};
\node (3) at (4.75, 11) {$32B_3$};
\node (7) at (6, 10) {$32A_1A_3$};
\node (12) at (0.5, 9) {$128A_3$};
\node (6) at (4, 9) {$64A_3$};
\node (15) at (2.5, 8) {$72A_1^2B_2$};
\node (1) at (2.65, 6.6) {$288A_1B_2$};
\node (2) at (4.3, 5.7) {$288B_2$};
\node (11) at (6, 5) {$512A_1A_2$};
\node (13) at (2.5, 4) {$1024A_2$};
\node (0) at (1.3, 4) {$72A_1^4$};
\node (9) at (1.5, 3) {$576A_1^3$};
\node (8) at (4.95, 3) {$576A_1^3$};
\node (18) at (5, 2) {$2304A_1^2$};
\node (14) at (1, 2) {$1152A_1^2$};
\node (5) at (3, 2) {$576A_1^2$};
\node (19) at (4, 1) {$768A_1$};
\node (10) at (2, 1) {$2304A_1$};
\node (4) at (3, 0) {$384$};
\draw (5) -- (9);
\draw (9) -- (1);
\draw (5) -- (6);
\draw (10) -- (5);
\draw (2) -- (1);
\draw (1) -- (15);
\draw (0) -- (16);
\draw (12) -- (16);
\draw (18) -- (8);
\draw (14) -- (9);
\draw (3) -- (17);
\draw (18) -- (1);
\draw (16) -- (17);
\draw (9) -- (0);
\draw (6) -- (7);
\draw (4) -- (19);
\draw (13) -- (12);
\draw (10) -- (18);
\draw (6) -- (3);
\draw (5) -- (8);
\draw (8) -- (15);
\draw (13) -- (11);
\draw (4) -- (10);
\draw (10) -- (13);
\draw (19) -- (18);
\draw (13) -- (6);
\draw (19) -- (2);
\draw (2) -- (3);
\draw (8) -- (7);
\draw (15) -- (17);
\draw (11) -- (7);
\draw (8) -- (3);
\draw (0) -- (15);
\draw (5) -- (2);
\draw (14) -- (12);
\draw (10) -- (14);
\draw (7) -- (17);
\draw (18) -- (11);
\draw (6) -- (16);
\end{tikzpicture} \\ 
\hline
$C$ & \begin{tikzpicture}[yscale=0.6, baseline=(current bounding box.center)]\ssmall
\node (3) at 	(3, 7) {$C_3$};
\node (8) at 	(3.5, 6) {$9A_1B_2$};
\node (0) at 	(3, 5) {$18B_2$};
\node (1) at 	(2, 4) {$32A_2$};
\node (6) at 	(4, 3) {$6A_1^3$};
\node (7) at 	(3, 2) {$72A_1^2$};
\node (5) at 	(4, 2) {$36A_1^2$};
\node (9) at 	(4, 1) {$72A_1$};
\node (4) at 	(2, 1) {$144A_1$};
\node (2) at 	(3, 0) {$48$};
\draw (4) -- (7);
\draw (1) -- (3);
\draw (6) -- (8);
\draw (2) -- (9);
\draw (5) -- (6);
\draw (7) -- (8);
\draw (8) -- (3);
\draw (5) -- (0);
\draw (0) -- (8);
\draw (9) -- (5);
\draw (9) -- (7);
\draw (4) -- (1);
\draw (2) -- (4);
\draw (4) -- (0);
\end{tikzpicture}\qquad\begin{tikzpicture}[yscale=0.55, baseline=(current bounding box.center)]\ssmall
\node (9) at (4.5, 13) {$C_4$};
\node (7) at (6.5, 12) {$16A_1C_3$};
\node (12) at (5, 11) {$32C_3$};
\node (10) at (4, 10) {$18B_2^2$};
\node (6) at (2, 9) {$128A_3$};
\node (0) at (8, 8) {$72A_1^2B_2$};
\node (19) at (6, 7) {$288A_1B_2$};
\node (16) at (3, 7) {$288A_1B_2$};
\node (14) at (5, 6) {$288B_2$};
\node (1) at (3.95, 5) {$512A_1A_2$};
\node (15) at (3.15, 4) {$1024A_2$};
\node (8) at (8, 4) {$24A_1^4$};
\node (3) at (7.75, 3) {$192A_1^3$};
\node (2) at (6, 3) {$576A_1^3$};
\node (18) at (2.5, 2) {$1152A_1^2$};
\node (17) at (5, 2) {$2304A_1^2$};
\node (13) at (7, 2) {$576A_1^2$};
\node (5) at (3.5, 1) {$2304A_1$};
\node (4) at (6, 1) {$768A_1$};
\node (11) at (4.75, 0) {$384$};
\draw (6) -- (9);
\draw (3) -- (8);
\draw (8) -- (0);
\draw (0) -- (7);
\draw (15) -- (1);
\draw (0) -- (10);
\draw (19) -- (0);
\draw (13) -- (2);
\draw (15) -- (12);
\draw (4) -- (13);
\draw (16) -- (10);
\draw (11) -- (5);
\draw (5) -- (18);
\draw (17) -- (2);
\draw (18) -- (16);
\draw (10) -- (9);
\draw (2) -- (16);
\draw (17) -- (19);
\draw (11) -- (4);
\draw (14) -- (19);
\draw (5) -- (17);
\draw (12) -- (7);
\draw (17) -- (1);
\draw (18) -- (6);
\draw (3) -- (19);
\draw (7) -- (9);
\draw (14) -- (16);
\draw (13) -- (14);
\draw (13) -- (3);
\draw (4) -- (17);
\draw (5) -- (14);
\draw (2) -- (0);
\draw (15) -- (6);
\draw (5) -- (15);
\draw (19) -- (12);
\draw (1) -- (7);
\end{tikzpicture} \\ 
\hline
$D$ & 
\begin{tikzpicture}[yscale=0.8, baseline=(current bounding box.center)]\ssmall
\node (7) at 	(2, 6) {$D_4$};
\node (10) at 	(3.25, 5) {$32A_3$};
\node (8) at 	(0.75, 5) {$32A_3$};
\node (0) at 	(2, 5) {$32A_3$};
\node (9) at 	(2, 4) {$512A_2$};
\node (5) at 	(1, 4) {$36A_1^4$};
\node (1) at 	(0.65, 3) {$288A_1^3$};
\node (11) at 	(0, 2) {$288A_1^2$};
\node (6) at 	(0.8, 2) {$288A_1^2$};
\node (2) at 	(4, 2) {$288A_1^2$};
\node (3) at 	(2, 1) {$1152A_1$};
\node (4) at 	(2, 0) {$192$};
\draw (11) -- (8);
\draw (9) -- (0);
\draw (3) -- (2);
\draw (8) -- (7);
\draw (10) -- (7);
\draw (9) -- (8);
\draw (5) -- (7);
\draw (3) -- (11);
\draw (2) -- (1);
\draw (6) -- (0);
\draw (9) -- (10);
\draw (1) -- (5);
\draw (11) -- (1);
\draw (0) -- (7);
\draw (3) -- (6);
\draw (4) -- (3);
\draw (3) -- (9);
\draw (2) -- (10);
\draw (6) -- (1);
\end{tikzpicture} \\
\hline
$F$ &
\begin{tikzpicture}[yscale=0.6, baseline=(current bounding box.center)]\ssmall
\node (14) at (4.5, 15) {$F_4$};
\node (19) at (1.75, 14) {$9B_4$};
\node (21) at (0.5, 13) {$6D_4$};
\node (16) at (4.5, 12) {$144A_1C_3$};
\node (12) at (1.5, 10.75) {$288B_3$};
\node (1) at (4.5, 10.75) {$288C_3$};
\node (4) at (3.3, 10) {$288A_1A_3$};
\node (17) at (0.7, 9) {$576A_3$};
\node (10) at (3.25, 8) {$648A_1^2B_2$};
\node (9) at (6.25, 8) {$512A_2^2$};
\node (11) at (3.75, 7) {$2592A_1B_2$};
\node (20) at (3.4, 6) {$2592B_2$};
\node (5) at (5, 5) {$4608A_1A_2$};
\node (3) at (6.6, 5) {$4608A_1A_2$};
\node (22) at (6.25, 4) {$3072A_2$};
\node (18) at (2.9, 4) {$3072A_2$};
\node (13) at (-0.25, 4) {$216A_1^4$};
\node (15) at (0, 3) {$1728A_1^3$};
\node (0) at (1, 3) {$5184A_1^3$};
\node (8) at (4, 2) {$20736A_1^2$};
\node (6) at (0.75, 2) {$5184A_1^2$};
\node (23) at (6, 1) {$6912A_1$};
\node (7) at (2, 1) {$6912A_1$};
\node (2) at (4, 0) {$1152$};
\draw (5) -- (9);
\draw (15) -- (13);
\draw (23) -- (22);
\draw (8) -- (0);
\draw (23) -- (8);
\draw (0) -- (10);
\draw (16) -- (14);
\draw (6) -- (17);
\draw (8) -- (5);
\draw (13) -- (21);
\draw (11) -- (10);
\draw (20) -- (12);
\draw (20) -- (11);
\draw (4) -- (19);
\draw (17) -- (21);
\draw (1) -- (16);
\draw (19) -- (14);
\draw (2) -- (23);
\draw (15) -- (11);
\draw (11) -- (1);
\draw (0) -- (4);
\draw (3) -- (16);
\draw (6) -- (15);
\draw (21) -- (19);
\draw (22) -- (3);
\draw (5) -- (4);
\draw (9) -- (14);
\draw (8) -- (11);
\draw (10) -- (19);
\draw (6) -- (0);
\draw (7) -- (18);
\draw (0) -- (12);
\draw (23) -- (20);
\draw (18) -- (5);
\draw (3) -- (9);
\draw (18) -- (17);
\draw (13) -- (10);
\draw (17) -- (4);
\draw (2) -- (7);
\draw (8) -- (3);
\draw (10) -- (16);
\draw (12) -- (19);
\draw (6) -- (20);
\draw (17) -- (12);
\draw (22) -- (1);
\draw (7) -- (8);
\draw (7) -- (6);
\end{tikzpicture} \\ 
\hline
$G$ & \begin{tikzpicture}[yscale=.75,baseline=(current bounding box.center)]\ssmall
\node (0) at (1, 4) {$G_2$};
\node (3) at (2, 3) {$2A_2$};
\node (1) at (0, 2) {$9A_1^2$};
\node (5) at (2, 1) {$18A_1$};
\node (4) at (0, 1) {$18A_1$};
\node (2) at (1, 0) {$12$};
\draw (3) -- (0);
\draw (2) -- (4);
\draw (5) -- (1);
\draw (2) -- (5);
\draw (4) -- (1);
\draw (1) -- (0);
\draw (5) -- (3);
\end{tikzpicture} \\ 
\hline
\end{longtable}


\begin{thebibliography}{ASM}

\bibitem[B]{bie97}
	R.~Bielawski,
	\textit{Hyperk\"ahler structures and group actions},
	J. Lond. Math. Soc. (2)
	\textbf{55} (1997), no. 2, 400--414. 

\bibitem[CG]{chr09}
	N.~Chriss, V.~Ginzburg,
	\textit{Representation Theory and Complex Geometry},
	Modern Birkh{\"a}user Classics,
	Birkh{\"a}user Boston, Boston, MA, 2010.

\bibitem[CL]{cod55}
	E.~A.~Coddington, N.~Levinson,
	\textit{Theory of Ordinary Differential Equations},
	McGraw-Hill Book Company, New York-Toronto-London, 1955.

\bibitem[CM]{col93}
	D.~H.~Collingwood, W.~M.~McGovern,
	\textit{Nilpotent Orbits in Semisimple Lie Algebras},
	Van Nostrand Reinhold Mathematics Series,
	Van Nostrand Reinhold,
	New York, 1993.

\bibitem[Da]{dan93}
	A.~Dancer,
	\textit{Dihedral singularities and gravitational instantons},
	J. Geom. Phys.
	\textbf{12} (1993), no. 2,  77--91.

\bibitem[DS1]{dan96}
	A.~Dancer, A.~Swann,
	\textit{Hyperk{\"a}hler metrics associated to compact Lie groups},
	Math. Proc. Cambridge Philos. Soc.
	\textbf{120} (1996), no. 1, 61--69.

\bibitem[DS2]{dan97}
	A.~Dancer, A.~Swann,
	\textit{The geometry of singular quaternionic K{\"a}hler quotients},
	Int. J. Math.
	\textbf{8} (1997), no. 5, 595--610.

\bibitem[DK]{dui12}
	J.~J.~Duistermaat, J.~A.~C.~Kolk,
	\textit{Lie Groups},
	Universitext,
	Springer-Verlag,
	Berlin, 2000.

\bibitem[Dy]{dyn72}
	E.~B.~Dynkin,
	\textit{Semisimple subalgebras of semisimple Lie algebras},
	Am. math. Soc. Transl. Ser. 2,
	\textbf{6} (1957), 111-245.

\bibitem[FC]{flo14}
	C.~Florentino, S.~Lawton,
	\textit{Topology of character varieties of Abelian groups},
	Topology Appl.
	\textbf{173} (2014), 32--58. 

\bibitem[HL]{hei94}
	P.~Heinzner, F.~Loose,
	\textit{Reduction of complex Hamiltonian G-spaces},
	Geom. Funct. Anal.
	\textbf{4} (1994), no. 3, 288--297. 

\bibitem[HKLR]{hit87}
	N.~J.~Hitchin, A.~Karlhede, U.~Lindstr{\"o}m, M.~Ro{\v{c}}ek,
	\textit{Hyperk{\"a}hler metrics and supersymmetry},
	Comm. Math. Phys.
	\textbf{108} (1987), no. 4, 535--589.

\bibitem[KN]{kem79}
	G.~Kempf, L.~Ness,
	\textit{The length of vectors in representation spaces},
	in: Algebraic geometry (Proc. Summer Meeting, Univ. Copenhagen, Copenhagen, 1978), pp. 233--243, 
	Lecture Notes in Math., Vol. 732,
	Springer, Berlin, 1979.

\bibitem[K]{kro88}
	P.~Kronheimer,
	\textit{A hyperk{\"a}hler structure on the cotangent bundle of a complex Lie group},
	MSRI preprint 1988, https://arxiv.org/abs/math/0409253.

\bibitem[L1]{lun73}
	D.~Luna,
	\textit{Slices \'etales},
	Bull. Soc. Math. France
	\textbf{33} (1973), 81--105.

\bibitem[L2]{lun76}
	D.~Luna,
	\textit{Fonctions diff\'erentiables invariantes sous l'op\'eration d'un groupe r\'eductif},
	Ann. Inst. Fourier (Grenoble),
	\textbf{26} (1976), no. 1, 33--49.

\bibitem[MW]{mar74}
	J.~Marsden, A.~Weinstein,
	\textit{Reduction of symplectic manifolds with symmetry},
	Rep. Math. Phys. \textbf{5} (1974), no. 1, 121--130.

\bibitem[OV]{on94}
	A.~L.~Onishchik, E.~B.~Vinberg,
	\textit{Lie Groups and Lie Algebras III},
	Encyclopaedia of Mathematical Sciences, Vol. 41,
	Springer-Verlag, Berlin, 1994.

\bibitem[PV]{pop94}
	V.~L.~Popov, E.~B.~Vinberg,
	\textit{Invariant Theory},
	in: Algebraic geometry IV, pp. 123--278,
	Encyclopaedia of Mathematical Sciences, Vol. 55, 
	Springer, Berlin Heidelberg, 1994.

\bibitem[Sc]{sch89}
	G.~W.~Schwarz,
	\textit{The topology of algebraic quotients},
	in: Topological methods in algebraic transformation groups (New Brunswick, NJ, 1988), pp. 135--151, 
	Progr. Math., Vol 80, 
	Birkh\"auser Boston, Boston, MA, 1989. 

\bibitem[Sj]{sja95}
	R.~Sjamaar,
	\textit{Holomorphic slices, symplectic reduction and multiplicities of representations},
	Ann. of Math.
	\textbf{141} (1995), no. 1, 87--129.

\bibitem[SL]{sja91}
	R.~Sjamaar, E.~Lerman,
	\textit{Stratified symplectic spaces and reduction},
	Ann. of Math.
	\textbf{134} (1991), no. 2, 375--422.


\end{thebibliography}
\end{document}